\documentclass[english]{article}
\usepackage[T1]{fontenc}
\usepackage[latin9]{inputenc}
\usepackage{mathtools}
\usepackage{amsthm}
\usepackage{amssymb}
\usepackage{stackrel}
\usepackage{xcolor}
\usepackage[margin = 3 cm,marginpar= 2cm]
{geometry}

\makeatletter

\newtheorem{thm}{Theorem}[section]
\newtheorem{prop}{Proposition}[section]
\newtheorem{lem}{Lemma}[section]

\newtheorem{rem}{Remark}[section]
\newtheorem{claim}{Claim}[section]

\numberwithin{equation}{section}

\@ifundefined{date}{}{\date{}}
\makeatother

\usepackage{babel}

\begin{document}
\title{The graph limit for a pairwise competition model}

\author{
Immanuel Ben-Porat%
     \thanks{Mathematical Institute, University of Oxford, Oxford OX2 6GG, UK.  {Immanuel.BenPorat@maths.ox.ac.uk}} %
\and
    Jos\'e A. Carrillo%
     \thanks{Mathematical Institute, University of Oxford, Oxford OX2 6GG, UK.  {mailto:carrillo@maths.ox.ac.uk}} %
     \and
    Pierre-Emmanuel Jabin%
     \thanks{Pennsylvania State University 109 McAllister University Park, PA 16802 US  {pejabin@psu.edu}{pejabin@psu.edu}} %
}

\maketitle

\begin{abstract}
This paper is aimed at extending the graph limit with time dependent
weights obtained in \cite{1} for the case of a pairwise competition
model introduced in \cite{10}, in which the equation governing the
weights involves a weak singularity at the origin. Well posedness
for the graph limit equation associated with the ODE system of the
pairwise competition model is also proved. 
\end{abstract}

\section{Introduction }

\textbf{General Background}. In this work, we are concerned with analyzing
the graph limit of the following system of $(d+1)N$ ODEs 

\begin{equation}
\left\{ \begin{array}{lc}
\dot{x_{i}}^{N}(t)=\frac{1}{N}\stackrel[j=1]{N}{\sum}m_{j}^{N}(t)\mathbf{a}(x_{j}^N(t)-x_{i}^N(t)),&\ x_{i}^{N}(0)=x_{i}^{0,N}\\
\dot{m}_{i}^{N}(t)=\psi_{i}^{N}(\mathbf{x}_{N}(t),\mathbf{m}_{N}(t)),&\ m_{i}^{N}(0)=m_{i}^{0,N}.
\end{array}\right.\label{eq:OPINION DYNAMICS}
\end{equation}
The notation is as follows: the unknowns are $x_{i}^N\in\mathbb{R}^{d}$
and $m_{i}^N\in\mathbb{R}$ are referred to as the opinions and weights
respectively. The evolution of the opinions is given in terms of the
weights and a function $\mathbf{a}:\mathbb{R}^{d}\rightarrow\mathbb{R}^{d}$
which is called the influence. The evolution of the weights is given
by means of functions $\psi_{i}^{N}:\mathbb{R}^{dN}\times\mathbb{R}^{N}\rightarrow\mathbb{R}$
where we apply the notation 

\[
\mathbf{x}_{N}(t)\coloneqq(x_{1}^{N}(t),...,x_{N}^{N}(t)),\ \mathbf{m}_{N}(t)\coloneqq(m_{1}^{N}(t),...,m_{N}^{N}(t)).
\]
This model has been proposed in \cite{10}, along with several other
models which are meant to idealize social dynamics. We refer to \cite{10,13}
for more details of how these models originate from biology and social
sciences. Mathematically, the
system (\ref{eq:OPINION DYNAMICS}) is a weighted version of the first
order $N-$body problem (simply by taking all the weights to be identically
equal to $1$). By
now, the mean field limit of the $N-$body problem 
\begin{equation}
\dot{x_{i}}^{N}(t)=\frac{1}{N}\stackrel[j=1]{N}{\sum}\mathbf{a}(x_{j}^N(t)-x_{i}^N(t)),\ x_{i}^{N}(0)=x_{i}^{0,N}\label{N body problem}
\end{equation}
is fairly well understood even for influence functions with strong singularities
at the origin \cite{14}. The mean field limit can be analysed in terms of the empirical measure defined by 
\[
\mu_{N}(t)\coloneqq\frac{1}{N}\stackrel[i=1]{N}{\sum}\delta_{x_{i}^N(t)}.
\]
Thanks to the work of Dobrushin \cite{4} it is possible to prove
quantitative convergence of $\mu_{N}(t)$ to the solution $\mu$ of
the (velocity free) Vlasov equation 

\begin{equation}
\partial_{t}\mu(t,x)-\mathrm{div}(\mu\mathbf{a}\star\mu)(t,x)=0,\ \mu(0,\cdot)=\mu^{0}\label{eq:homogenous transport}
\end{equation}
 with respect to the Wasserstein metric (provided this is true initially
of course). The mean field limit with time dependent weights has been
investigated in \cite{1,5,6} for Lipschitz continuous
interactions and $\psi_{i}^{N}$ which are
at least Lipschitz in each variable, and more recently in \cite{2} 
for the case of the 1D attractive Coulomb interaction (but still with
$\psi_{i}^{N}$ regular enough). There is a different regime, the so called graph limit, closely related to the mean-field limit. In the graph limit, we pass from a discrete system
of ODEs to a ``continuous'' system in the following sense: we associate
to $\mathbf{x}_{N}(t),\mathbf{m}_{N}(t)$ the following Riemman sums
$\widetilde{x}_{N}:[0,T]\times I\rightarrow\mathbb{R},\widetilde{m}_{N}:[0,T]\times I\rightarrow\mathbb{R}$
defined by 

\[
\widetilde{x}_{N}(t,s)\coloneqq\stackrel[i=1]{N}{\sum}x_{i}(t)\mathbf{1}_{[\frac{i-1}{N},\frac{i}{N}]}(s),\ \widetilde{m}_{N}(t,s)\coloneqq\stackrel[i=1]{N}{\sum}m_{i}(t)\mathbf{1}_{[\frac{i-1}{N},\frac{i}{N}]}(s).
\]
Using the equation for the trajectories of the opinions and weights,
one easily finds that $\widetilde{x}_{N},\widetilde{m}_{N}$ are governed
by the following equations
\[
\left\{ \begin{array}{lc}
\displaystyle
\partial_{t}\widetilde{x}_{N}(t,s)=\int_{I}\widetilde{m}_{N}(t,s_{\ast})\mathbf{a}(\widetilde{x}_{N}(t,s_{\ast})-\widetilde{x}_{N}(t,s))ds_{\ast}, & \widetilde{x}_{N}(0,s)=\widetilde{x}_{N}^{0}(s)\\
\displaystyle
\partial_{t}\widetilde{m}_{N}(t,s)=N\int_{\frac1N \lfloor sN\rfloor}^{\frac1N (\lfloor sN\rfloor +1)}
\psi(s_{\ast},\widetilde{x}_{N}(t,\cdot),\widetilde{m}_{N}(t,\cdot))ds_{\ast}, & \widetilde{m}_{N}(0,s)=\widetilde{m}_{N}^{0}(s).
\end{array}\right.
\]
Lebesgue differentiation theorem leads us formally to the following
integro-differential equation
\begin{equation}
\left\{ \begin{array}{lc}
\displaystyle
\partial_{t}x(t,s)=\int_{I} m(t,s_{\ast})\mathbf{a}(x(t,s_{\ast})-x(t,s))ds_{\ast}, & x(0,s)=x^{0}(s)\\
\partial_{t}m(t,s)=\Psi(s,x(t,\cdot),m(t,\cdot)), & m(0,s)=m^{0}(s).
\end{array}\right.\label{eq:integro differential system intro}
\end{equation}
Here $\Psi:I\times L^{\infty}(\mathbb{R}^{d})\times L^{\infty}(\mathbb{R}^{d})\rightarrow\mathbb{R}$
is a functional whose relation to $\psi_{i}^{N}$ is given by the
formula (\ref{eq:formula for psi}) in the next section. The formula
relating $x^{0}(s),m^{0}(s)$ to $\widetilde{x}_{N}^{0}(s),\widetilde{m}_{N}^{0}(s)$
will be given in the next section as well (formula (\ref{initial data})).
Hence, one expects that the sums $\widetilde{x}_{N}(t,s),\widetilde{m}_{N}(t,s)$
are an approximation of the solution $(x(t,s),m(t,s))$ of the Equation
(\ref{eq:integro differential system intro}) . 

Before going further, let us briefly comment on the origin of the
terminology ``graph limit''. This name stems from the fact that
the system (\ref{eq:OPINION DYNAMICS}) can be viewed as a nonlinear
heat equation on a graph. For example, in the case where the weights
are time independent and the $\mathbf{a}$ is taken to be the identity,
then the system (\ref{eq:OPINION DYNAMICS}) can be rewritten as the
linear heat equation with respect to the Laplacian associated to the
underlying simple graph. This is the point of view which has been
taken in \cite{11}. However, this underlying combinatorial structure
seems to come into play mostly when the weights may vary from one
opinion to another, in which case methods from graph theory prove
as highly useful. We also refer to the more recent work \cite{9} for a demonstration of the power of graph theory techniques in the context of the mean field limit, and \cite{3} in the context of convergence to consensus
for the graph limit equation. See also \cite{12} for a proof of the graph limit for metric valued labels, alongside an extensive explanation of the relation between the graph limit and the hydrodynamic and mean field limits.  In our settings, which are very similar
to the framework in \cite{1}, this graph structure is not as relevant,
and we shall therefore not dwell on this matter. It is instructive
to view the system (\ref{eq:integro differential system intro}) as
continuous version of (\ref{eq:OPINION DYNAMICS}), in the sense that
it is obtained by replacing averaged sums by integrals on the unit
interval and summation indices by variables in the unit interval.

\medskip{}

\textbf{Relevant Literature and Contribution of the Present work}.
It appears that the graph limit point of view has not received as
much attention as the mean-field limit. The study of this problem
was initiated in \cite{11}, which as already remarked, considers
time independent weights which may depend on the index of the opinion
as well. This result has been extended in \cite{1} to cover time
dependent weights (although in \cite{1} the weights depend only on
the summation index). The evolution in time of the weights renders difficult
the problem both at the microscopic and graph limit level- since the
corresponding ODE/integro-differential equation become coupled (compare
for instance Equations \eqref{eq:OPINION DYNAMICS} and \eqref{N body problem}),
and at the macroscopic level- since the mean field PDE includes a
non-local source term (see Section 4 for more details). In both of
these results, the functions $\psi_{i}^{N}$ are assumed to be well
behaved in terms of regularity.
On the other hand, models corresponding to scenarios
where the functions $\psi_{i}^{N}$ exhibit singularities recently received attention in \cite{10}. For instance, the following ODE has been
studied in \cite{10}: 
\begin{equation}
\left\{ \begin{array}{lc}
\dot{x}_{i}^{N}(t)=\frac{1}{N}\stackrel[j=1]{N}{\sum}m_{j}^{N}(t)\mathbf{a}(x_{j}^{N}(t)-x_{i}^{N}(t)), & x_{i}^{N}(0)=x_{i}^{0,N}\\
\dot{m}_{i}^{N}(t)=\frac{1}{N}\stackrel[j=1]{N}{\sum}m_{i}^{N}(t)m_{j}^{N}(t)\left\langle \frac{\dot{x}_{i}^{N}(t)+\dot{x}_{j}^{N}(t)}{2},\mathbf{s}(x_{i}^{N}(t)-x_{j}^{N}(t))\right\rangle, & \ m_{i}^{N}(0)=m_{i}^{0,N},
\end{array}\right. \label{eq:pairwise comp}
\end{equation}
where $\mathbf{a}:\mathbb{R}^{d}\rightarrow\mathbb{R}^{d}$ is Lipschitz
and takes the form $\mathbf{a}(x)=a(\left|x\right|)x$ for some radial
$a:\mathbb{R}\rightarrow\mathbb{R}$, and $\mathbf{s}:\mathbb{R}^{d}\rightarrow\mathbb{S}^{d-1}$
is the projection on the unit sphere, i.e. 
\begin{equation}
\mathbf{s}(x)\coloneqq\left\{ \begin{array}{c}
\frac{x}{\left\Vert x\right\Vert },\ x\neq0\\
0,\ x=0.
\end{array}\right.\label{eq:definition of s}
\end{equation}
Of course, inserting the equation for $x_{i}$ into the equation for
$m_{i}$ transfers the system to the form (\ref{eq:OPINION DYNAMICS}).
System (\ref{eq:pairwise comp}) is referred to as a pairwise
competition model in \cite{10}, and its well posedness can be proved
provided opinions are separated initially ($i\neq j\Longrightarrow x_{i}^{0}\neq x_{j}^{0}$).
It is the aim of this work to investigate how to overcome the challenges
created due to the singularity in the weight function in the context of the graph limit.
The problem of the graph limit for singularities in the influence function is also interesting. As already remarked, for the mean field limit this has been successfully achieved in \cite{2} for the 1D attractive Coulomb case. However, it is not clear how to study the graph limit regime in this whole generality. The 1D repulsive Coulomb interaction however is manageable, and can be handled by similar methods to the one demonstrated in
the present work. 

A first contribution of the present work is reflected on two levels, both of which are considered in 1D: the well posedness of the  graph limit equation (\ref{eq:integro differential system intro}), and the derivation of  (\ref{eq:integro differential system intro}) from the opinion dynamics (\ref{eq:OPINION DYNAMICS}) in the limit as $N\rightarrow\infty$. As for the first point, we
note that in the case when $\mathbf{a}$ and $\psi_{i}^{N}$ are well behaved then equation (\ref{eq:integro differential system intro}) can be viewed as a Banach valued ODE, and noting that at each time $t$ our unknowns $(x(t,\cdot),m(t,\cdot))$ are functions of the variable $s$, and therefore there is a straightforward analogy between the well posedness of the discrete System (\ref{eq:OPINION DYNAMICS})
and equation (\ref{eq:integro differential system intro}). As
already mentioned, the global well-posedness of the finite dimensional
version of Equation (\ref{eq:integro differential system intro}),
namely System (\ref{eq:pairwise comp}) has been (among other things)
proved in \cite{10} using the theory of differential inclusions as
developed by Fillipov \cite{7}. Originally, Fillipov formulated his
theory for unknowns taking values in a finite dimensional space in
contrast to Equation (\ref{eq:integro differential system intro}).
We follow a slightly different route which is in fact more elementary and does not require any familiarity with convex analysis. A second contribution of the present work, is studying the graph limit in arbitrary dimensions $d>1$. In higher dimensions, a natural assumption to impose on the initial datum $x^0$ is that it is bi-Lipschitz in $s$ - an assumption of this type is strictly stronger from what is needed in 1D. This in turn leads to considering Riemann sums whose labeling variable $s$ varies on the $d$-dimensional unit cube rather on the unit interval, because cubes of different dimensions cannot be diffeomorphic. This labelling procedure does not have any modelling interpretation since particles (opinions) are still exchangeable or indistinguishable. It would in fact be possible to still work on the unit interval through a change of variable on the labeling variable, since all cubes (and most measurable spaces that one may use) are isomorphic to the unit interval per the Borel isomorphism theorem. However the corresponding analysis would be far more convoluted, and instead having the labeling variable on the $d$-dimensional unit cube make the various technical steps more transparent. These considerations are therefore detailed separately in Section \ref{Section 5}. For both points
it is crucial to observe the lower bound $\left|x(t,s_{2})-x(t,s_{1})\right|\gtrsim\left|x^{0}(s_{2})-x^{0}(s_{1})\right|$. In 1D the initial separation at the continuous
level will be replaced by the assumption that $x^{0}$ is increasing, whereas in higher dimensions this assumption will be replaced by requiring that $x^0$ is bi-Lipschitz. Finally we remark that the method here extends the case of the main results in \cite{1}, in the sense that it simultaneously covers functions $\mathbf{s}$ which are either Lipschitz or have a jump discontinuity at the origin. This last observation is simple but not obvious- for example in the case where the singularity emerges from the influence part, as mentioned earlier, it is not clear how to unify both results. 
\medskip{}

We organize the paper as follows: Section \ref{sec:Preliminaries}
reviews the terminology introduced in \cite{1} in the specific context
of system (\ref{eq:pairwise comp}). In particular, Section \ref{sec:Preliminaries}
includes preliminaries such as the existence and uniqueness of classical
solutions to the system (\ref{eq:OPINION DYNAMICS}) in the present
settings and other basic properties of solutions (of course, uniqueness
is not strictly needed for the purpose of the graph or mean field
limit). Section \ref{sec:Well-Posedness-for-1} is a continuous adaptation
of section \ref{sec:Preliminaries}, namely well posedness for the 1D graph limit equation for which uniqueness is essential. Section \ref{sec:The-Gronwall-Estimate}
includes the main evolution estimate leading to the 1D graph limit, and clarifies the link between the mean field and the graph limit. In Section \ref{Section 5} we introduce multi-dimensional Riemann sums and study the graph limit for arbitrary $d>1$.    

\section{\label{sec:Preliminaries}Preliminaries }

\subsection{The ODE system. }

Recall that the system which will occupy us is 
\begin{equation}
\left\{ \begin{array}{lc}
\dot{x}_{i}^{N}(t)=\frac{1}{N}\stackrel[j=1]{N}{\sum}m_{j}(t)\mathbf{a}(x_{j}(t)-x_{i}(t)) & x_{i}(0)=x_{i}^{0,N}\\
\dot{m}_{i}^{N}(t)=\psi_{i}^{N}(\mathbf{x}_{N}(t),\mathbf{m}_{N}(t)) & m_{i}(0)=m_{i}^{0,N}
\end{array}\right.\label{opinion dynamics sec2}
\end{equation}
where 

\begin{equation}
\psi_{i}^{N}(\mathbf{x}_{N},\mathbf{m}_{N})\coloneqq\frac{1}{2N^{2}}m_{i}\underset{j,k}{\sum}m_{j}m_{k}\left(\mathbf{a}(x_{k}-x_{i})+\mathbf{a}(x_{k}-x_{j})\right)\mathbf{s}(x_{i}-x_{j}).\label{FORMULA FOR PSIN}
\end{equation}
When $d=1$, which is the case of main interest here, we note that $\mathbf{s}$ identifies with the sign function. We start by reviewing the well-posedness theory which has been established
for the System (\ref{opinion dynamics sec2}) in \cite{10}. As usual
with ODEs with weakly singular right hand sides, the argument in \cite{10}
rests on the theory of differential inclusions as developed by
Fillipov \cite{7} and the fact that opinions remain separated for
all times provided this is true initially. Unless necessary, we omit the super index $N$ in the opinions and weights. 
\begin{prop}
\textup{\label{well posedness of opinion dynamics } (\cite[Proposition
3]{10} } Suppose $\mathbf{a}:\mathbb{R}^{d}\rightarrow\mathbb{R}^{d}$ is Lipschitz
with $\mathbf{a}(0)=0$ and $x_{i}^{0}\neq x_{j}^{0}$ for all $i\neq j$.
Then there exists a unique classical solution $(\mathbf{x}_{N}(t),\mathbf{m}_{N}(t))$
to the System (\ref{opinion dynamics sec2}) with $x_{i}(t)\neq x_{j}(t)$
for all $i\neq j$ and $t\geq0$. 
\end{prop}
We also recap the following basic properties of solutions, which already
appear implicitly or explicitly in \cite{1,10}, and will
appear in the course of the proof of the main theorems. 
\begin{lem}
\label{basic properties } Let the assumptions of Proposition \ref{well posedness of opinion dynamics }
hold. Assume also $m_{i}^{0}>0,i=1,...,N$ and $\mathrm{Lip}(\mathbf{a})=L$. Let $(\mathbf{x}_{N}(t),\mathbf{m}_{N}(t))$
be the solution of System (\ref{opinion dynamics sec2}) on $[0,T]$.
Then

i. (Conservation of total mass). $\frac{1}{N}\stackrel[i=1]{N}{\sum}m_{i}^{0}=1\Longrightarrow\frac{1}{N}\stackrel[i=1]{N}{\sum}m_{i}(t)=1,\ t\in[0,T]$. 

ii. (Uniform bound in time on opinions). If $\left|x_{i}^{0}\right|\leq\overline{X}$
then for all $t\in[0,T]$ it holds that 

\[
\left|x_{i}(t)\right|\leq\overline{X}e^{2LT}.
\]

iii. (Uniform bound in time on weights). $m_{i}(t)>0$ for all $t\in[0,T]$
with the estimate 
\[
m_{i}^{\mathrm{0}}e^{-2\overline{X}e^{2LT}t}\leq m_{i}(t)\leq m_{i}^{\mathrm{0}}e^{2\overline{X}e^{2LT}t}.
\]

iv. (Opinions are separated). There is a constant $C=C(L,T)>1$ such
for all $t\in[0,T]$ the following bound holds 
\[
\frac{1}{C}\left|x_{i}^{0}-x_{j}^{0}\right|\leq\left|x_{i}(t)-x_{j}(t)\right|\leq C\left|x_{i}^{0}-x_{j}^{0}\right|.
\]
\end{lem}
\textit{Proof}. For i. see Proposition 2 in \cite{10}. For ii., fix
a time $\tau>0$ such that $m_{j}(\tau)\geq0,j=1,...,N$ for all $t\in[0,\tau]$
(such a time exists by continuity). We utilize i. and the assumption
$\mathbf{a}(0)=0$ to find that for each $t\in[0,\tau_{0}]$ 

\[
\left|x_{i}(t)\right|\leq\left|x_{i}^{0}\right|+\frac{2L}{N}\stackrel[j=1]{N}{\sum}\int_{0}^{t}\left|m_{j}(\tau)\right|\underset{1\leq k\leq N}{\max}\left|x_{k}(\tau)\right|d\tau=\overline{X}+2L\int_{0}^{t}\underset{1\leq k\leq N}{\max}\left|x_{k}(\tau)\right|d\tau
\]
so that 

\[
\underset{1\leq k\leq N}{\max}\left|x_{k}(t)\right|\leq\overline{X}+2L\int_{0}^{t}\underset{1\leq k\leq N}{\max}\left|x_{k}(\tau)\right|d\tau,
\]
which by Gronwall's Lemma implies 

\begin{equation}
\underset{1\leq k\leq N}{\max}\left|x_{k}(t)\right|\leq\overline{X}e^{2LT}.\label{eq:-5}
\end{equation}
We prove iii., from which we will conclude ii. for all $t\in[0,T]$.
We start by explaining why $m_{i}(t)>0$. Indeed, if on the contrary
$m_{i}(t)\leq0$ for some $1\leq i\leq N$ and $t\in[0,T]$ and let 

\[
\tau\coloneqq\inf\left\{ t\in[0,T]\left|\exists1\leq i\leq N:m_{i}(t)\leq0\right.\right\} .
\]
Then the bound from ii. and preservation of total mass of i. imply
that for all $t\in[0,\tau)$ we have 

\[
\left|\frac{d}{dt}\frac{1}{2}\log\left(m_{i}^{2}\right)\right|=\left|\frac{\dot{m_{i}}(t)}{m_{i}}\right|=\left|\frac{1}{2N^{2}}\underset{j,k}{\sum}m_{j}m_{k}\left(\mathbf{a}(x_{k}-x_{i})+\mathbf{a}(x_{k}-x_{j})\right)\mathbf{s}(x_{i}-x_{j})\right|\leq2\overline{X}e^{2LT},
\]
hence 

\[
-2\overline{X}e^{2LT}\leq\frac{1}{2}\frac{d}{dt}\log\left(m_{i}^{2}(t)\right)\leq2\overline{X}e^{2LT}.
\]
Integration in time yields that for all $t\in[0,\tau]$

\[
-2\overline{X}e^{2LT}t+\log\left(m_{i}^{\mathrm{0}}\right)\leq\log\left(m_{i}\right)\leq2\overline{X}e^{2LT}t+\log\left(m_{i}^{\mathrm{0}}\right),
\]
and consequently 
\[
m_{i}^{\mathrm{0}}e^{-2\overline{X}e^{2LT}t}\leq m_{i}(t)\leq m_{i}^{\mathrm{0}}e^{2\overline{X}e^{2LT}t}.
\]
Letting $t\nearrow\tau$ yields a contradiction. Therefore $m_{i}(t)>0$
for all $t\in[0,T]$ which in turn implies that (\ref{eq:-5}) holds
for all $t\in[0,T]$. Remark also that the same estimate done on the
interval $[0,T]$ yields the asserted bound on $[0,T]$. Point iv.
is Proposition 7 in \cite{10}. 
\begin{flushright}
$\square$
\par\end{flushright}

\subsection{The graph limit equation. \label{graph limit eq 1D} }

In the graph limit we attach to the flow of System (\ref{eq:OPINION DYNAMICS})
the following ``Riemman sums'' 

\begin{equation}
\widetilde{x}_{N}(t,s)\coloneqq\stackrel[i=1]{N}{\sum}x_{i}(t)\mathbf{1}_{[\frac{i-1}{N},\frac{i}{N}]}(s),\ \widetilde{m}_{N}(t,s)\coloneqq\stackrel[i=1]{N}{\sum}m_{i}(t)\mathbf{1}_{[\frac{i-1}{N},\frac{i}{N}]}(s).\label{Riemman sums}
\end{equation}
The functional $\Psi:I\times L^{\infty}(I)\times L^{\infty}(I)\rightarrow\mathbb{R}$
and the functions $x^{0}:I\rightarrow\mathbb{R}^{d},m^{0}:I\rightarrow\mathbb{R}$
are given and the functions $\psi_{i}^{N}$ and the initial data
$x_{i}^{0,N},m_{i}^{0,N}$ are defined in terms of these functions
through the following formula

\begin{equation}
\psi_{i}^{N}(\mathbf{x}_{N}(t),\mathbf{m}_{N}(t))\coloneqq N\int_{_{\frac{i-1}{N}}}^{^{\frac{i}{N}}}\Psi(s,\widetilde{x_{N}}(t,\cdot),\widetilde{m_{N}}(t,\cdot))ds\label{eq:formula for psi}
\end{equation}
and 

\begin{equation}
x_{i}^{0,N}\coloneqq N\int_{_{\frac{i-1}{N}}}^{^{\frac{i}{N}}}x^{0}(s)ds,\ m_{i}^{0,N}\coloneqq N\int_{_{\frac{i-1}{N}}}^{^{\frac{i}{N}}}m^{0}(s)ds.\label{initial data}
\end{equation}
If $\Psi$ is given by 

\[
\Psi(s,x(\cdot),m(\cdot))\coloneqq m(s)\iint_{I^{2}}m(s_{\ast})m(s_{\ast\ast})\left(\mathbf{a}(x(s_{\ast\ast})-x(s))+\mathbf{a}(x(s_{\ast\ast})-x(s_{\ast}))\right)\mathbf{s}(x(s)-x(s_{\ast}))ds_{\ast}ds_{\ast\ast},
\]
then one readily checks that the $\psi_{i}^{N}$ in Formula (\ref{FORMULA FOR PSIN})
are recovered via Formula (\ref{eq:formula for psi}). Notice that
by Lebesgue's differentiation theorem $\widetilde{x}_{N}(0,s),\widetilde{m}_{N}(0,s)$
well approximate $x^{0}(s),m^{0}(s)$ because for a.e. $s$ we have
pointwise convergence 
\[
\widetilde{x}_{N}(0,s)=N\int_{\frac{\left\lfloor sN\right\rfloor }{N}}^{\frac{\left\lfloor sN\right\rfloor +1}{N}}x^{0}(\sigma)d\sigma\underset{N\rightarrow\infty}{\rightarrow}x^{0}(s),\thinspace\widetilde{m}_{N}(0,s)=N\int_{\frac{\left\lfloor sN\right\rfloor }{N}}^{\frac{\left\lfloor sN\right\rfloor +1}{N}}m^{0}(\sigma)d\sigma\underset{N\rightarrow\infty}{\rightarrow}m^{0}(s).
\]
Also, it is worthwhile remarking that unlike in the mean field limit
regime, where the initial data realizing the initial convergence can
be chosen from a set of full measure, here we use a very specific
choice for the initial data, and in particular all initial data of
the form specified by formula (\ref{initial data}) constitute a set
of measure $0$, which means that the probabilistic methods that we
have at our disposal in the mean field limit become useless in the
graph limit. We will return to this point in Section \ref{sec:The-Gronwall-Estimate}.
The functions $\widetilde{x}_{N}(t,s),\widetilde{m}_{N}(t,s)$ defined
through Formula (\ref{Riemman sums}) are governed by the following
equations, which should be compared with the graph limit Equation
(\ref{eq:integro differential system intro}).
\begin{prop}
\label{equation governing xN,mN} Let the assumptions of Proposition
\ref{well posedness of opinion dynamics } hold and let $(\mathbf{x}_{N}(t),\mathbf{m}_{N}(t))$
be the solution to System (\ref{opinion dynamics sec2}) on $[0,T]$.
Let $\widetilde{x}_{N},\widetilde{m}_{N}$ be given by (\ref{Riemman sums}).
Then 

\begin{equation}
\left\{ \begin{array}{cc}
\partial_{t}\widetilde{x}_{N}(t,s)=\int_{I}\widetilde{m}_{N}(t,s_{\ast})\mathbf{a}(\widetilde{x}_{N}(t,s_{\ast})-\widetilde{x}_{N}(t,s))ds_{\ast},\\[2mm]
\partial_{t}\widetilde{m}_{N}(t,s)=N\displaystyle\int_{\frac{\left\lfloor sN\right\rfloor }{N}}^{\frac{\left\lfloor sN\right\rfloor +1}{N}}\Psi(s_{\ast},\widetilde{x}_{N}(t,\cdot),\widetilde{m}_{N}(t,\cdot))ds_{\ast}.  
\end{array}\right.\label{eq:labled equation governing xN,mN}
\end{equation}
\end{prop}
\begin{proof}
We start with the equation for $\widetilde{x}_{N}(t,s)$. Fix $s\in\left[\frac{i_{0}-1}{N},\frac{i_{0}}{N}\right)$, we get
\begin{align*}
\partial_{t}\widetilde{x}_{N}(t,s)=&\,\stackrel[i=1]{N}{\sum}\dot{x}_{i}(t)\mathbf{1}_{[\frac{i-1}{N},\frac{i}{N}]}(s)=\frac{1}{N}\stackrel[i=1]{N}{\sum}\stackrel[j=1]{N}{\sum}m_{j}(t)\mathbf{a}(x_{j}(t)-x_{i}(t))\mathbf{1}_{[\frac{i-1}{N},\frac{i}{N}]}(s)\\
=&\frac{1}{N}\stackrel[j=1]{N}{\sum}m_{j}(t)\mathbf{a}(x_{j}(t)-x_{i_{0}}(t))=\frac{1}{N}\stackrel[j=1]{N}{\sum}m_{j}(t)\mathbf{a}(x_{j}(t)-\widetilde{x}_{N}(t,s)).
\end{align*}
On the other hand, we have
\begin{align*}
\int_{I}\widetilde{m}_{N}(t,s_{\ast})\mathbf{a}(\widetilde{x}_{N}(t,s_{\ast})-\widetilde{x}_{N}(t,s))ds_{\ast}=&\int_{I}\stackrel[j=1]{N}{\sum}m_{j}(t)\mathbf{1}_{[\frac{j-1}{N},\frac{j}{N}]}(s_{\ast})\mathbf{a}\left(\stackrel[k=1]{N}{\sum}x_{k}(t)\mathbf{1}_{[\frac{k-1}{N},\frac{k}{N}]}(s_{\ast})-\widetilde{x}_{N}(t,s)\right)ds_{\ast}\\
=\,&\stackrel[j=1]{N}{\sum}\int_{I}\mathbf{1}_{[\frac{j-1}{N},\frac{j}{N}]}(s_{\ast})m_{j}(t)\mathbf{a}\left(x_{j}(t)-\widetilde{x}_{N}(t,s)\right)ds_{\ast}\\=&\frac{1}{N}\stackrel[j=1]{N}{\sum}m_{j}(t)\mathbf{a}\left(x_{j}(t)-\widetilde{x}_{N}(t,s)\right). 
\end{align*}
The equation for $\widetilde{m}_{N}$ is obtained due to the following identities
\begin{align*}
\partial_{t}\widetilde{m}_{N}(t,s)=\stackrel[i=1]{N}{\sum}\dot{m}_{i}(t)\mathbf{1}_{[\frac{i-1}{N},\frac{i}{N}]}(s)=\,&N\stackrel[i=1]{N}{\sum}\mathbf{1}_{[\frac{i-1}{N},\frac{i}{N}]}(s)\int_{\frac{i-1}{N}}^{\frac{i}{N}} \Psi(s_{\ast},\widetilde{x}_{N}(t,\cdot),\widetilde{m}_{N}(t,\cdot))ds_{\ast}
\\=\,&N\int_{\frac{\left\lfloor sN\right\rfloor }{N}}^{\frac{\left\lfloor sN\right\rfloor +1}{N}}\Psi(s_{\ast},\widetilde{x}_{N}(t,\cdot),\widetilde{m}_{N}(t,\cdot))ds_{\ast}.
\end{align*}

\end{proof}

\section{\label{sec:Well-Posedness-for-1}Well Posedness for the Graph Limit
Equation }

\subsection{The decoupled equation }

We first decouple the equation and prove well posedness for the two resulting equations
separately as in \cite{1}. To be more precise, the system considered in \cite[Section 4]{1} reads 
\begin{equation}
\left\{ \begin{array}{lc}
\partial_{t}x(t,s)=\int_{I}m(t,s_{\ast})\mathbf{a}(x(t,s_{\ast})-x(t,s))ds_{\ast}, & x(0,s)=x^{0}(s)\\
\partial_{t}m(t,s)=\psi_{S,k}(s,x(t,\cdot),m(t,\cdot)), & m(0,s)=m^{0}(s),
\end{array}\right.\label{eq:graph limit AD version}
\end{equation}
where $\psi_{S,k}(s,x(t,\cdot),m(t,\cdot))\coloneqq m(t,s)\int_{I^{k}}m^{\otimes k}(s_{1},...,s_{k})S(x(s),x(s_{1}),...,x(s_{k}))ds_{1}...ds_{k}.$
The hypothesis on the function $S$ and the initial data $x^{0},m^{0}$
in \eqref{eq:graph limit AD version} assumed in \cite{1} are:
\begin{itemize}
    \item[\textbf{H1'}] $d\geq1,\mathbf{a}(0)=0$ and $\mathbf{a}\in\mathrm{Lip}(\mathbb{R}^{d})$.

     \item[\textbf{H2'}] $(x^{0},m^{0})\in L^{\infty}(I;\mathbb{R}^{d})\times L^{\infty}(I;\mathbb{R}_{>0})$. 

      \item[\textbf{H3'}] The function $S\in C_{b}(\mathbb{R}^{(k+1)d})\cap\mathrm{Lip}(\mathbb{R}^{(k+1)d})$
and there are $(i,j)\in\left\{ 0,...,k\right\} $ such that 
\begin{equation}
S(...,y_{i},...,y_{j},...)=-S(...,y_{j},...,y_{i},...).\label{pair condition}
\end{equation}
\end{itemize}
The most restrictive assumption for the graph limit is \textbf{H3'} since $S$ is not Lipschitz in problems of interest  mentioned in the introduction, see \cite{5,1}.
Furthermore, any solution to the graph limit equation \eqref{eq:integro differential system intro} is expected to satisfy an
estimate analogue to Inequality iv. in Lemma \ref{basic properties },
namely 
\[
\frac{1}{C}\left|x^{0}(s_{1})-x^{0}(s_{2})\right|\leq\left|x(t,s_{1})-x(t,s_{2})\right|\leq C\left|x^{0}(s_{1})-x^{0}(s_{2})\right|,
\]
which would imply Lipschitz continuity along the trajectories provided
$x^{0}$ is one to one. This also leads us to remark that the initial
separation in the microscopic system \eqref{eq:pairwise comp} can be replaced by the assumption
that $x^{0}$ is one to one in the infinite dimensional case, which
means that we need to be able to evaluate $x^{0}$ pointwise, and
therefore a more natural assumption is $x^{0}\in C(I)$ rather than
$x^{0}\in L^{\infty}$. 

To summarize, in contrast to \cite{1}, we assume the hypotheses:
\begin{itemize}
 \item[\textbf{H1}] $d=1$, $\mathbf{a}(0)=0$  and $\mathbf{a}\in\mathrm{Lip}(\mathbb{R})$
with $L\coloneqq\mathrm{Lip}(\mathbf{a})$. 

\item[\textbf{H2}] i. $m^{0}\in L^{\infty}(I)$, $\int_{I} m_{0}(s)ds=1$ and $\frac{1}{M}\leq m^{0}\leq M$ for some $M>1$.  

ii. $x^{0}\in C(I)$ is one to one and $\left|x^{0}\right|\leq X$ for some $X>0$.

 \item[\textbf{H3}]  i. The restrictions $\left.\mathbf{s}\right|_{(0,\infty)}$ and $\left.\mathbf{s}\right|_{(-\infty,0)}$ are Lipschitz, i.e. there is some $\mathbf{S}>0$ such that 
 \[ \left|\mathbf{s}(x_{1})-\mathbf{s}(x_{2})\right|\leq\mathbf{S}\left|x_{1}-x_{2}\right|,\ x_{1},x_{2}\in(0,\infty)\]
and 
\[\left|\mathbf{s}(x_{1}')-\mathbf{s}(x_{2}')\right|\leq\mathbf{S}\left|x_{1}'-x_{2}'\right|,\ x_{1}',x_{2}'\in(-\infty,0).\]
ii. $\mathbf{s}$ is odd ($\mathbf{s}(-x)=-\mathbf{s}(x)$) and there is some $\mathbf{S}_{\infty}>0$ such that $\left|\mathbf{s}(x)\right|\leq\mathbf{S}_{\infty},\ x\in\mathbb{R}.$

\end{itemize}
Clearly, the sign function is a particular example of hypothesis \textbf{H3}. In the following Lemma, which is a variant of \cite[Lemma 3]{1},
the new considerations discussed above will be taken into account. 

\begin{lem}
\label{Estimate for psi} Let hypothesis  \textbf{H1}-\textbf{H3} hold. 

1. Suppose that 

$\bullet\ m_{1},m_{2}\in C([0,T],L^{\infty}(I))$ are non-negative
and $\int_{I}m_{2}(t,s)ds=\int_{I}m_{1}(t,s)ds=1$ for all $t\in[0,T]$. 

$\bullet\ x\in C([0,T]\times I)$ with $\sup_{[0,T]\times I}\left|x\right|\leq\overline{X}$.

Then, for all $t\in[0,T]$ it holds that 

\[
\int_{I}\left|\Psi(s,x(t,\cdot),m_{1}(t,\cdot))-\Psi(s,x(t,\cdot),m_{2}(t,\cdot))\right|ds\leq12L\mathbf{S}_{\infty}\overline{X}\left\Vert m_{1}(t,\cdot)-m_{2}(t,\cdot)\right\Vert _{1}.
\]
2. Suppose that 

$\bullet\ m\in C([0,T],L^{\infty}(I))$ is non-negative and $\int_{I}m(t,s)ds=1$
for all $t\in[0,T]$ and 
\[\underset{[0,T]}{\sup}\left\Vert m(t,\cdot)\right\Vert _{\infty}\leq\overline{M}.\]

$\bullet\ x_{1},x_{2}\in C([0,T]\times I)$ are such that for all
$t\in[0,T]$ the maps $s\mapsto x_{1}(t,s),s\mapsto x_{2}(t,s)$ are

increasing. 

Then, for all $t\in[0,T]$ it holds that 
\[
\int_{I}\left|\Psi(s,x_{1}(t,\cdot),m(t,\cdot))-\Psi(s,x_{2}(t,\cdot),m(t,\cdot))\right|ds\leq L\left(3\overline{M}\mathbf{S}_{\infty}+\mathbf{S}_{\infty}+16\mathbf{S}\overline{X}\right)\sup_{I}\left|x_{1}(t,\cdot)-x_{2}(t,\cdot)\right|.
\]
\end{lem}
\textit{Proof}. \textbf{Step 1}. For readability, we suppress the time
variable (unless unavoidable). Set $\mathbf{a}(s,s_\ast,s_{\ast\ast}):=\mathbf{a}(x(s_{\ast\ast})-x(s))+\mathbf{a}(x(s_{\ast\ast})-x(s_{\ast})$, we have 
\begin{align}
&\left|m_{1}(s)\iint_{I^{2}}\right. m_{1}(s_{\ast})m_{1}(s_{\ast\ast})\mathbf{a}(s,s_\ast,s_{\ast\ast})\mathbf{s}(x(s)-x(s_{\ast}))ds_{\ast}ds_{\ast\ast}
\nonumber\\
   &\qquad\left.-m_{2}(s)\iint_{I^{2}}m_{2}(s_{\ast})m_{2}(s_{\ast\ast})\mathbf{a}(s,s_\ast,s_{\ast\ast})\mathbf{s}(x(s)-x(s_{\ast}))ds_{\ast}ds_{\ast\ast}\right|
\nonumber\\
   \leq \,& m_{1}(s)\left|\iint_{I^{2}}\left(m_{1}(s_{\ast})m_{1}(s_{\ast\ast})-m_{2}(s_{\ast})m_{2}(s_{\ast\ast})\right)\mathbf{a}(s,s_\ast,s_{\ast\ast})\mathbf{s}(x(s)-x(s_{\ast}))ds_{\ast}ds_{\ast\ast}\right| 
\nonumber\\
   &+\left|m_1(s)-m_{2}(s)\right|\left|\iint_{I^{2}}m_{2}(s_{\ast})m_{2}(s_{\ast\ast})\mathbf{a}(s,s_\ast,s_{\ast\ast})\mathbf{s}(x(s)-x(s_{\ast}))ds_{\ast}ds_{\ast\ast}\right|
\nonumber\\
\leq &\, 4L\mathbf{S}_{\infty}\sup_{[0,T]\times I}\left|x\right|m_{1}(s)\iint_{I^{2}}\left|m_{1}(s_{\ast})m_{1}(s_{\ast\ast})-m_{2}(s_{\ast})m_{2}(s_{\ast\ast})\right|ds_{\ast}ds_{\ast\ast}
\nonumber\\
&+4L\mathbf{S}_{\infty}\sup_{[0,T]\times I}\left|x\right|\left|m_{1}(s)-m_{2}(s)\right|\iint_{I^{2}}m_{2}(t,s_{\ast})m_{2}(t,s_{\ast\ast})ds_{\ast}ds_{\ast\ast}.\label{eq:-1}
\end{align}
Using the assumption that $\int_{I}m_{1}(s)ds=\int_{I}m_{2}(s)ds=1$,
the first integral in the right hand side of (\ref{eq:-1}) can be estimated as
\begin{align*}
\leq\iint_{I^{2}}\left|m_{1}(s_{\ast})m_{1}(s_{\ast\ast})-m_{1}(s_{\ast})m_{2}(s_{\ast\ast})\right|ds_{\ast}ds_{\ast\ast}+\iint_{I^{2}}\left|m_{1}(s_{\ast})m_{2}(s_{\ast\ast})-m_{2}(s_{\ast})m_{2}(s_{\ast\ast})\right|ds_{\ast}ds_{\ast\ast}\\  
=\int_{I}\left|m_{1}(s_{\ast\ast})-m_{2}(s_{\ast\ast})\right|ds_{\ast\ast}+\int_{I}\left|m_{1}(s_{\ast})-m_{2}(s_{\ast})\right|ds_{\ast}=2\left\Vert m_{1}(t,\cdot)-m_{2}(t,\cdot)\right\Vert _{1}.
\end{align*}
Therefore, integrating (\ref{eq:-1}) in $s$ over $I$ produces 
\begin{align*}
\int_{I}\left|\Psi(s,x(t,\cdot),m_{1}(t,\cdot))-\Psi(s,x(t,\cdot),m_{2}(t,\cdot))\right|ds\leq 12L\mathbf{S}_{\infty}\sup_{[0,T]\times I}\left|x\right|\left\Vert m_{1}(t,\cdot)-m_{2}(t,\cdot)\right\Vert _{1}.
\end{align*}

\textbf{Step 2}. Set $\mathbf{a}_i(s,s_\ast,s_{\ast\ast}):=\mathbf{a}(x_i(s_{\ast\ast})-x_i(s))+\mathbf{a}(x_i(s_{\ast\ast})-x_i(s_{\ast}))$, $i=1,2$, we can also estimate as 
\begin{align*}
\left|\iint_{I^{2}}m(s_{\ast})m(s_{\ast\ast})(\mathbf{a}_1(s,s_\ast,s_{\ast\ast})\mathbf{s}(x_{1}(s)-x_{1}(s_{\ast}))ds_{\ast}ds_{\ast\ast}-\mathbf{a}_2(s,s_\ast,s_{\ast\ast})\mathbf{s}(x_{2}(s)-x_{2}(s_{\ast})))ds_{\ast}ds_{\ast\ast}\right|\hspace{3.0 cm}\\
\leq L\mathbf{S}_{\infty}\iint_{I^{2}}m(s_{\ast})m(s_{\ast\ast})\left(2\left|x_{1}(s_{\ast\ast})-x_{2}(s_{\ast\ast})\right|+\left|x_{1}(s_{\ast})-x_{2}(s_{\ast})\right|+\left|x_{1}(s)-x_{2}(s)\right|\right)ds_{\ast}ds_{\ast\ast}\hspace{3.0 cm}\\
+\iint_{I^{2}}m(s_{\ast})m(s_{\ast\ast})\left|\mathbf{a}_2(s,s_\ast,s_{\ast\ast})\right| \left|\mathbf{s}(x_{1}(s)-x_{1}(s_{\ast}))-\mathbf{s}(x_{2}(s)-x_{2}(s_{\ast}))\right|ds_{\ast}ds_{\ast\ast} \hspace{4.0 cm}\\ 
\coloneqq J_{1}(t,s)+J_{2}(t,s).\hspace{4.0 cm}
\end{align*}
Since $s\mapsto x_{1}(t,s),s\mapsto x_{2}(t,s)$ are increasing, we recognize from assumption $\mathbf{H3}$ that 
\begin{align*}
J_{2}(t,s)=&\int_{I}\int_{0}^{s}m(s_{\ast})m(s_{\ast\ast})\left|\mathbf{a}_{2}(s,s_{\ast},s_{\ast\ast})\right|\left|\mathbf{s}(x_{1}(s)-x_{1}(s_{\ast}))-\mathbf{s}(x_{2}(s)-x_{2}(s_{\ast}))\right|ds_{\ast}ds_{\ast\ast}\\
&+\int_{I}\int_{s}^{1}m(s_{\ast})m(s_{\ast\ast})\left|\mathbf{a}_{2}(s,s_{\ast},s_{\ast\ast})\right|\left|\mathbf{s}(x_{1}(s)-x_{1}(s_{\ast}))-\mathbf{s}(x_{2}(s)-x_{2}(s_{\ast}))\right|ds_{\ast}ds_{\ast\ast}
\\\leq&8L\overline{X}\mathbf{S}\underset{I}{\sup}\left|x_{1}(t,\cdot)-x_{2}(t,\cdot)\right|+8L\overline{X}\mathbf{S}\underset{I}{\sup}\left|x_{1}(t,\cdot)-x_{2}(t,\cdot)\right|\\=&16L\overline{X}\mathbf{S}\underset{I}{\sup}\left|x_{1}(t,\cdot)-x_{2}(t,\cdot)\right|.
\end{align*}
We estimate  $\int_{I}m(t,s)\left|J_{1}(t,s)\right|ds$. 
\begin{align*}
\int_{I} m(t,s)\left|J_{1}(t,s)\right|ds\leq &\,3L\mathbf{S}_{\infty}\int_{I}m(s)\underset{[0,T]}{\sup}\left\Vert m(t,\cdot)\right\Vert _{\infty}\underset{I}{\sup}\left|x_{1}(t,\cdot)-x_{2}(t,\cdot)\right|ds\\&+L\mathbf{S}_{\infty}\int_{I}m(s)\underset{I}{\sup}\left|x_{1}(t,\cdot)-x_{2}(t,\cdot)\right|ds\\=&\,L\mathbf{S}_{\infty}\left(3\underset{[0,T]}{\sup}\left\Vert m(t,\cdot)\right\Vert _{\infty}+1\right)\underset{I}{\sup}\left|x_{1}(t,\cdot)-x_{2}(t,\cdot)\right|.  
\end{align*}
As a result, we obtain 
\begin{align*}
\int_{I}\left|\Psi(s,x_{1}(t,\cdot),m(t,\cdot))-\Psi(s,x_{2}(t,\cdot),m(t,\cdot))\right|ds\leq&\int_{I}m(s)\left|J_{1}(t,s)\right|ds+\int_{I}m(s)\left|J_{2}(t,s)\right|ds\\
\leq &\, L\left(3\overline{M}\mathbf{S}_{\infty}+\mathbf{S}_{\infty}+16\mathbf{S}\overline{X}\right)\sup_{I}\left|x_{1}(t,\cdot)-x_{2}(t,\cdot)\right|.
\end{align*}

\begin{flushright}
$\square$
\par\end{flushright}
\begin{lem}
\label{decoupled eq} Let hypotheses \textbf{H1}-\textbf{H3} hold. Suppose also 
\begin{itemize}
    \item[$\bullet$] $\overline{x}\in C([0,T]\times I)$ is such that for each $t\in[0,T]$
the map $s\mapsto\overline{x}(t,s)$ is one to one. 

     \item[$\bullet$] $\overline{m}\in C([0,T];L^{\infty}(I))$
is non-negative such that $\int_{I}\overline{m}(t,s)ds=1$ for each $t\in[0,T]$.
\end{itemize}

Then, there exists a unique solution $(x,m)\in C^{1}([0,T];C(I))\oplus C^{1}([0,T];L^{\infty}(I))$
to the decoupled system 

\begin{equation}
\left\{ \begin{array}{lc}
\partial_{t}x(t,s)=\displaystyle{\int}_{I}\overline{m}(t,s_{\ast})\mathbf{a}(x(t,s_{\ast})-x(t,s))ds_{\ast}, &\ x(0,s)=x^{0}(s)\\
\partial_{t}m(t,s)=\Psi(s,\overline{x}(t,\cdot),m(t,\cdot)), &\ m(0,s)=m^{0}(s).
\end{array}\right.\label{DECOUPLED EQUATION}
\end{equation}
The solution $x$ is such that $s\mapsto x(t,s)$ is one to one and
the solution $m$ is non-negative such that $\int_{I}m(t,s)ds=1$. 
\end{lem}
\textbf{Step 1}. \textbf{Existence and uniqueness for the equation
for $x$}. Fix $0<\underline{T}<\frac{1}{2L\left\Vert \overline{m}\right\Vert _{\infty,\infty}}$.
Let $M_{x_{0}}$ be the metric space of functions in $C([0,\underline{T}]\times I)$
with $x(0,s)=x^{0}(s)$. Define the operator $K_{x_{0}}:M_{x_{0}}\rightarrow C([0,\underline{T}]\times I)$
by 
\[
K_{x_{0}}(x)(t,s)\coloneqq x^{0}(s)+\int_{0}^{t}\int_{I}\overline{m}(\tau,s_{\ast})\mathbf{a}(x(\tau,s_{\ast})-x(\tau,s))ds_{\ast}d\tau.
\]
We view $M_{x_{0}}$ as a complete metric space. We then have
\begin{align*}
\left|(K_{x_{0}}x)(t,s)-(K_{x_{0}}y)(t,s)\right|=&\left|\int_{0}^{t}\int_{I}\overline{m}(\tau,s_{\ast})\left(\mathbf{a}(x(\tau,s_{\ast})-x(\tau,s))-\mathbf{a}(y(\tau,s_{\ast})-y(\tau,s))\right)ds_{\ast}d\tau\right|\\ 
\leq&\left|\int_{0}^{t}\int_{I}\overline{m}(\tau,s_{\ast})\left(\mathbf{a}(x(\tau,s_{\ast})-x(\tau,s))-\mathbf{a}(x(\tau,s_{\ast})-y(\tau,s))\right)ds_{\ast}d\tau\right|\\&\,+\left|\int_{0}^{t}\int_{I}\overline{m}(\tau,s_{\ast})\left(\mathbf{a}(x(\tau,s_{\ast})-y(t,s))-\mathbf{a}(y(\tau,s_{\ast})-y(t,s))\right)ds_{\ast}d\tau\right|\\\leq&\, 2L\left\Vert \overline{m}\right\Vert _{\infty,\infty}\underline{T}\underset{I\times[0,\underline{T}]}{\sup}\left|x-y\right|.
\end{align*}
The choice of $\underline{T}$ ensures $2L\left\Vert \overline{m}\right\Vert _{\infty,\infty}\underline{T}<1$, thereby
making the Banach contraction principle available which implies there
exist a unique solution $x\in C([0,\underline{T}]\times I)$ to the
equation 

\[
x(t,s)=x^{0}(s)+\int_{0}^{t}\int_{I}\overline{m}(\tau,s_{\ast})\mathbf{a}(x(\tau,s_{\ast})-x(\tau,s))ds_{\ast}d\tau.
\]
By a standard iteration argument we have existence and uniqueness
on the whole interval $[0,T]$. Evidently the map $\tau\mapsto\int_{I}\overline{m}(\tau,s_{\ast})\mathbf{a}(x(\tau,s_{\ast})-x(\tau,s))ds_{\ast}$
is continuous so that by the fundamental theorem of calculus we conclude
$x\in C^{1}([0,\underline{T}];C(I))$. Next we claim that this solution
must be one to one. 
\begin{claim}
\label{opinions remain seperated } Let $x\in C^{1}([0,T];C(I))$
be a solution of  
\[
\partial_{t}x(t,s)=\int_{I}\overline{m}(t,s_{\ast})\mathbf{a}(x(t,s_{\ast})-x(t,s))ds_{\ast}.
\]
Then for all $t\in[0,T]$ and all $s_{1},s_{2}\in I$ it hold that 

\[
\left|x(t,s_{2})-x(t,s_{1})\right|^{2}\geq e^{-Lt}\left|x^{0}(s_{2})-x^{0}(s_{1})\right|^{2}.
\]
In particular, $s\mapsto x(t,s)$ is increasing.  
\end{claim}
\textit{Proof}. We start by showing that $\left|x(t,s_{2})-x(t,s_{1})\right|^{2}>0$
for all $t\in[0,T]$. Assume to the contrary there is some $t\in[0,T]$
and $s_{2}>s_{1}$ such that $x(t,s_{2})=x(t,s_{1})$ and set 

\[
\tau_{0}\coloneqq\inf\left\{ t\in[0,T]\left|x(t,s_{2})=x(t,s_{1})\right.\right\} >0.
\]
Then for all $t\in[0,\tau_{0})$ we have $\left|x(t,s_{2})-x(t,s_{1})\right|>0$
and as a result 
\[
\frac{d}{dt}\left|x(t,s_{2})-x(t,s_{1})\right|^{2}=\left(x(t,s_{2})-x(t,s_{1})\right)\underset{I}{\int}m(t,s_{\ast})\left(\mathbf{a}(x(t,s_{\ast})-x(t,s_{2}))-\mathbf{a}(x(t,s_{\ast})-x(t,s_{1}))\right)ds_{\ast}
\]

\begin{equation}
\geq-L\left|x(t,s_{2})-x(t,s_{1})\right|^{2}\underset{I}{\int}m(t,s_{\ast})ds_{\ast}=-L\left|x(t,s_{2})-x(t,s_{1})\right|^{2}.\hspace{-1.5 cm}\label{eq:-2}
\end{equation}
Division by $\left|x(t,s_{2})-x(t,s_{1})\right|^{2}\neq0$ implies
\[
\frac{d}{dt}\log\left(\left|x(t,s_{2})-x(t,s_{1})\right|^{2}\right)\geq-L,
\]
which in turn gives the inequality 

\[
\left|x(t,s_{2})-x(t,s_{1})\right|^{2}\geq e^{-Lt}\left|x^{0}(s_{2})-x^{0}(s_{1})\right|^{2}>0.
\]
Taking $t\nearrow\tau_{0}$ gives a contradiction. Repeating now the
estimate (\ref{eq:-2}) shows that in fact for all $t\in[0,T]$ 
\[
\left|x(t,s_{2})-x(t,s_{1})\right|^{2}\geq e^{-Lt}\left|x^{0}(s_{2})-x^{0}(s_{1})\right|^{2}.
\]
By continuity and the assumption that $x^{0}$ is increasing it follows that $s\mapsto x(t,s)$ is  increasing. 

\textbf{Step 2}. \textbf{Existence and uniqueness for the equation
for $m$}.\\ 
\textbf{2.1}. \textbf{Short time}.  Put $\mathbf{X}\coloneqq\underset{I\times[0,T]}{\sup}\left|\overline{x}\right|+1$
and pick $0<\underline{T}\leq\frac{1}{16L\mathbf{S}_{\infty}\mathbf{X}M^{4}} $.
Let $M_{m_{0}}$ be the metric subspace of $C([0,\underline{T}];L^{1}(I))$
of functions with $m(0,s)=m^{0}(s)$, $0\leq m\leq2M$ and $\int_{I}m(t,s)ds=1$
for all $t\in[0,\underline{T}]$. Define $K_{m_{0}}:M_{m_{0}}\rightarrow C([0,\underline{T}];L^{1}(I))$
by 

\[
K_{m_{0}}(m)(t,s)\coloneqq m^{0}(s)+\int_{0}^{t}\Psi(s,\overline{x}(\tau,\cdot),m(\tau,\cdot))d\tau.
\]
We start by observing that $K$ maps $M_{m_{0}}$ into itself. 
\begin{claim}
$0\leq K_{m_{0}}(m)\leq2M$ and $\int_{I}K_{m_{0}}(m)(t,s)ds=1$ for
all $t\in[0,\underline{T}]$. 
\end{claim}
\textit{Proof}. To see why $K_{m_{0}}(m)$ is non-negative notice
that because how $\underline{T}$ was chosen we have 

\[
K_{m_{0}}(m)\geq\frac{1}{M}-4L\mathbf{S}_{\infty}\underline{T}\mathbf{X}M^{3}>\frac{1}{2M}.
\]
Moreover 
\[
K_{m_{0}}(m)\leq M+4L\mathbf{S}_{\infty}\underline{T}\mathbf{X}M^{3}\leq2M.
\]
To show that $K_{m_{0}}(m)$ has unit integral we use that $\mathbf{s}$
is odd. 
\begin{align*}
\int_{I}\int_{0}^{t}&\Psi(s,\overline{x}(\tau,\cdot),m(\tau,\cdot))d\tau ds\hspace{10.0 cm}\\=&\frac{1}{2}\int_{0}^{t}\int_{I^{3}}m(\tau,s)m(\tau,s_{\ast})m(\tau,s_{\ast\ast})\mathbf{a}(x(\tau,s_{\ast\ast})-x(\tau,s))\mathbf{s}(\overline{x}(\tau,s)-\overline{x}(\tau,s_{\ast}))dsds_{\ast}ds_{\ast\ast}d\tau\\  
&+\frac{1}{2}\int_{0}^{t}\int_{I^{3}}m(\tau,s)m(\tau,s_{\ast})m(\tau,s_{\ast\ast})\mathbf{a}(x(\tau,s_{\ast\ast})-x(\tau,s_{\ast}))\mathbf{s}(\overline{x}(\tau,s)-\overline{x}(\tau,s_{\ast}))dsds_{\ast}ds_{\ast\ast}d\tau.
\end{align*}
Changing variables $s\longleftrightarrow s^{\ast}$ and using that
$\mathbf{s}$ is odd, the second integral in the right hand side is
recast as 
\begin{align*}
\frac{1}{2}\int_{0}^{t}\int_{I^{3}}m(\tau,s)m(\tau,s_{\ast})m(\tau,s_{\ast\ast})\mathbf{a}(x(\tau,s_{\ast\ast})-x(\tau,s_{\ast}))\mathbf{s}(\overline{x}(\tau,s_{\ast})-\overline{x}(\tau,s))dsds_{\ast}ds_{\ast\ast}d\tau\\
=-\frac{1}{2}\int_{0}^{t}\int_{I^{3}}m(\tau,s)m(\tau,s_{\ast})m(\tau,s_{\ast\ast})\mathbf{a}(x(\tau,s_{\ast\ast})-x(\tau,s))\mathbf{s}(\overline{x}(\tau,s)-\overline{x}(\tau,s_{\ast}))d\tau,
\end{align*}
which shows that 
\[
\int_{I}\int_{0}^{t}\Psi(s,\overline{x}(\tau,\cdot),m(\tau,\cdot))d\tau ds=0,
\]
as wanted.
\begin{flushright}
$\square$
\par\end{flushright}
We view $M_{m_{0}}$ as a complete metric space. Let $m,n\in M_{m_{0}}$.
Thanks to point 1. in Lemma \ref{Estimate for psi} we have 
\begin{align*}
\int_{I}\left|K_{m_{0}}(m)(t,s)-K_{m_{0}}(n)(t,s)\right|ds\leq&\int_{0}^{t}\int_{I}\left|\Psi(s,\overline{x}(\tau,\cdot),m(\tau,\cdot))-\Psi(s,\overline{x}(\tau,\cdot),n(\tau,\cdot))\right|dsd\tau\\
\leq&12L\mathbf{S}_{\infty}\underline{T}\sup_{[0,T]\times I}\left|\overline{x}\right|\left\Vert m(t,\cdot)-n(t,\cdot)\right\Vert _{1},
\end{align*}
and thus 
\[
\underset{t\in[0,T]}{\sup}\int_{I}\left|K_{m_{0}}(m)(t,s)-K_{m_{0}}(n)(t,s)\right|ds\leq12L\mathbf{S}_{\infty}\underline{T}\sup_{[0,T]\times I}\left|\overline{x}\right|\underset{[0,T]}{\sup}\left\Vert m(t,\cdot)-n(t,\cdot)\right\Vert _{1}.
\]
The choice of $\underline{T}$ makes the Banach contraction theorem
available thereby ensuring the existence of a unique solution $m\in M_{m_{0}}$
on $[ 0,\underline{T}]$ to the equation 

\[
m(t,s)=m^{0}(s)+\int_{0}^{t}\Psi(s,\overline{x}(\tau,\cdot),m(\tau,\cdot))d\tau.
\] 
Moreover, from the choice of $\underline{T}>0$  we evidently have 
\begin{align*}
m(t,s)\geq\frac{1}{2M}>0,\ t\in[0,\underline{T}].   \end{align*}
\textbf{2.2}. \textbf{long time}. 
 Let $m(t,s)$ be the unique solution on $[0,\underline{T}]$ to 
\begin{align*}
\partial_{t}m(t,s)=m(t,s)\iint_{I^{2}}m(t,s_{\ast})m(t,s_{\ast\ast})\left(\mathbf{a}(x(s_{\ast\ast})-x(s))+\mathbf{a}(x(s_{\ast\ast})-x(s_{\ast}))\right)\mathbf{s}(x(s)-x(s_{\ast}))ds_{\ast}ds_{\ast\ast}   
\end{align*}
given by step 2.1. Then we obtain 
\begin{align}\label{aux}
-4L\mathbf{X}\mathbf{S}_{\infty}m(t,s)\leq\partial_{t}m(t,s)\leq4L\mathbf{X}\mathbf{S}_{\infty}m(t,s)    
\end{align}
and as a result we deduce that
\begin{align*}
\left|\frac{d}{dt}\log(m(t,s))\right|\leq4L\mathbf{X}\mathbf{S}_{\infty}    
\end{align*}
i.e. 
\begin{align*}
 \frac{1}{\exp\left(4L\mathbf{X}\mathbf{S}_{\infty}T\right)M}\leq m(t,s)\leq M\exp\left(4L\mathbf{X}\mathbf{S}_{\infty}T\right).   
\end{align*}
Put  $\tau=\tau(L,\mathbf{X},\mathbf{S}_{\infty},M,T)=\exp\left(4L\mathbf{X}\mathbf{S}_{\infty}T\right)M$. Then we get a solution on $\left[0,2\times\frac{1}{16L\mathbf{X}\mathbf{S}_{\infty}\tau}\right].$
Iterating the process $k>16LT\mathbf{X}\mathbf{S}_{\infty}\tau$ times we get existence and uniqueness of a solution on $[0,T]$.
We claim now to have the upgrade $m\in C^{1}([0,T];L^{\infty}(I))$. Indeed, we have
\[
\left|m(t,s)\right|\leq\left|m^{0}(s)\right|+2\sup_{[0,T]\times I}\left|\overline{x}\right|L\mathbf{S}_{\infty}\int_{0}^{t}\left|m(\tau,s)\right|d\tau,
\]
so that 
\[
\left\Vert m(t,\cdot)\right\Vert _{\infty}\leq\left\Vert m^{0}\right\Vert _{\infty}+2L\mathbf{S}_{\infty}\sup_{[0,T]\times I}\left|\overline{x}\right|\int_{0}^{t}\left\Vert m(\tau,\cdot)\right\Vert _{\infty}d\tau
\]
which entails 
\[
\left\Vert m(t,\cdot)\right\Vert _{\infty}\leq\left\Vert m^{0}\right\Vert _{\infty} e^{2L\mathbf{S}_{\infty}\sup_{[0,T]\times I}\left|\overline{x}\right|t}
\]
which upon maximizing in $t$ yields 
\begin{equation}
\sup_{[0,T]}\left\Vert m(t,\cdot)\right\Vert _{\infty}\leq\left\Vert m^{0}\right\Vert _{\infty} e^{2L\mathbf{S}_{\infty}\sup_{[0,T]\times I}\left|\overline{x}\right|T}.\label{Uniform bound for the weights}
\end{equation}
Taking into account \eqref{aux}, we finally conclude $m\in C^{1}([0,T];L^{\infty}(I))$.
\begin{flushright}
$\square$
\par\end{flushright}

\subsection{The coupled equation }

The well posedness for the decoupled equation serves as the main tool
for proving well posedness of the original system. We prove 
\begin{thm}
Let hypothesis \textbf{H1}-\textbf{H3} hold. There exist a unique solution $(x,m)\in C^{1}([0,T];C(I))\oplus C^{1}([0,T];L^{\infty}(I))$
to the system 
\begin{equation}
\left\{ \begin{array}{lc}
\partial_{t}x(t,s)=\int_{I}m(t,s_{\ast})\mathbf{a}(x(t,s_{\ast})-x(t,s))ds_{\ast},&\ x(0,s)=x^{0}(s)\\
\partial_{t}m(t,s)=\Psi(s,x(t,\cdot),m(t,\cdot)),&\ m(0,s)=m^{0}(s).
\end{array}\right.\label{eq:COUPLED EQUATION}
\end{equation}
\end{thm}
\textit{Proof}. \textbf{Step 1. Existence}. We define recursively
the following sequence of functions $(x_{n},m_{n})$: 

i. For all $t\in[0,T]$ and all $s\in I$ we set $x_{0}(t,s)=x^{0}(s)$
and for all $t\in[0,T]$ and a.e. $s\in I$ we set $m_{0}(t,s)=m^{0}(s)$. 

ii. If $(x_{n-1},m_{n-1})$ have been defined we define $(x_{n},m_{n})$
to be the unique solution guaranteed by Lemma \ref{decoupled eq}
to the equation 

\[
\left\{ \begin{array}{lc}
\partial_{t}x_{n}(t,s)=\int_{I}m_{n-1}(t,s_{\ast})\mathbf{a}(x_{n}(t,s_{\ast})-x_{n}(t,s))ds_{\ast},&\ x_{n}(0,s)=x_{0}(s)\\
\partial_{t}m_{n}(t,s)=\Psi(s,x_{n-1}(t,\cdot),m_{n}(t,\cdot)),&\ m_{n}(0,s)=m_{0}(s).
\end{array}\right.
\]
Start by noting that $x_{n}$ is uniformly bounded (with respect to
$n)$ in the space $C([0,T]\times I)$. We have 

\[
\left|x_{n}(t,s)\right|=x^{0}(s)+\int_{0}^{t}\int_{I}m_{n-1}(\tau,s_{\ast})\mathbf{a}(x_{n}(\tau,s_{\ast})-x_{n}(\tau,s))ds_{\ast}d\tau\leq\underset{I}{\sup}\left|x^{0}\right|+2L\int_{0}^{t}\underset{I}{\sup}\left|x_{n}(\tau,\cdot)\right|d\tau,
\]
so that 

\[
\underset{I}{\sup}\left|x_{n}(t,\cdot)\right|\leq e^{2LT}\underset{I}{\sup}\left|x^{0}\right|,
\]
hence 

\[
\underset{I\times[0,T]}{\sup}\left|x_{n}(t,s)\right|\leq e^{2LT}\underset{I}{\sup}\left|x^{0}\right|\leq \overline{X}(X,L,T)\coloneqq\overline{X}.
\]
This also implies a uniform bound in $n$ for the weights since in
view of Inequality (\ref{Uniform bound for the weights}) 

\[
\left\Vert m_{n}(t,\cdot)\right\Vert _{\infty}\leq\left\Vert m^{0}\right\Vert _{\infty}e^{2\overline{X}T}\leq \overline{M}(M,X,L,T)\coloneqq\overline{M}.
\]
The proof of existence essentially boils down to proving that $(x_{n},m_{n})$
is a Cauchy sequence in the space $C([0,T];C(I))\oplus C([0,T];L^{1}(I))$.\\ 

\noindent\textit{Estimate for $\underset{I}{\sup}\left|x_{n+1}(t,\cdot)-x_{n}(t,\cdot)\right|$:}
\begin{align}
\left|x_{n+1}(t,s)-x_{n}(t,s)\right|=&\left|\int_{0}^{t}\int_{I}m_{n}(\tau,s_{\ast})\mathbf{a}(x_{n+1}(\tau,s_{\ast})-x_{n+1}(\tau,s))ds_{\ast}d\tau\right.\nonumber\\&\left.-\int_{0}^{t}\int_{I}m_{n-1}(\tau,s_{\ast})\mathbf{a}(x_{n}(\tau,s_{\ast})-x_{n}(\tau,s))ds_{\ast}d\tau\right|   \nonumber
\\\leq&\left|\int_{0}^{t}\int_{I}m_{n}(\tau,s_{\ast})\left(\mathbf{a}(x_{n+1}(\tau,s_{\ast})-x_{n+1}(\tau,s))-\mathbf{a}(x_{n}(\tau,s_{\ast})-x_{n}(\tau,s))\right)ds_{\ast}d\tau\right|\nonumber\\
&+\left|\int_{0}^{t}\int_{I}\left(m_{n}(\tau,s_{\ast})-m_{n-1}(\tau,s_{\ast})\right)\mathbf{a}(x_{n}(\tau,s_{\ast})-x_{n}(\tau,s))ds_{\ast}d\tau\right|\nonumber
\end{align}
\begin{align}
\leq&\, L\int_{0}^{t}\int_{I}m_{n}(\tau,s_{\ast})\left|x_{n+1}(\tau,s_{\ast})-x_{n}(\tau,s_{\ast})\right|ds_{\ast}d\tau\nonumber\\
&\,+L\int_{0}^{t}\int_{I}m_{n}(\tau,s_{\ast})\left|x_{n+1}(\tau,s)-x_{n}(\tau,s)\right|ds_{\ast}d\tau\nonumber
\\&\, +\int_{0}^{t}\int_{I}\left|m_{n}(t,s_{\ast})-m_{n-1}(t,s_{\ast})\right|\left|\mathbf{a}(x_{n}(t,s_{\ast})-x_{n}(t,s))\right|ds_{\ast}dt\nonumber\\
\leq&\, L\int_{0}^{t}\int_{I}m_{n}(\tau,s_{\ast})\left|x_{n+1}(\tau,s_{\ast})-x_{n}(\tau,s_{\ast})\right|ds_{\ast}d\tau\nonumber\\
&+L\int_{0}^{t}\int_{I}m_{n}(\tau,s_{\ast})\left|x_{n+1}(\tau,s)-x_{n}(\tau,s)\right|ds_{\ast}d\tau\nonumber\\
&+\int_{0}^{t}\int_{I}\left|m_{n}(t,s_{\ast})-m_{n-1}(t,s_{\ast})\right|\left|\mathbf{a}(x_{n}(t,s_{\ast})-x_{n}(t,s))\right|ds_{\ast}dt\nonumber\\
\leq&\, 2L\int_{0}^{t}\underset{I}{\sup}\left|x_{n+1}(\tau,\cdot)-x_{n}(\tau,\cdot)\right|d\tau+2\overline{X}L\int_{0}^{t}\left\Vert m_{n}(\tau,\cdot)-m_{n-1}(\tau,\cdot)\right\Vert _{1}d\tau\nonumber
\\
\leq&\, C_{1}(\overline{X},L)\int_{0}^{t}\underset{I}{\sup}\left|x_{n+1}(\tau,\cdot)-x_{n}(\tau,\cdot)\right|+\left\Vert m_{n}(\tau,\cdot)-m_{n-1}(\tau,\cdot)\right\Vert _{1}d\tau.\hspace{-2.0 cm}\label{ESTIMATE FOR xn}
\end{align}

\noindent\textit{Estimate for} $\left\Vert m_{n+1}(t,\cdot)-m_{n}(t,\cdot)\right\Vert _{1}.$
\begin{align*}
\left|m_{n+1}(t,s)-m_{n}(t,s)\right|\leq&\left|\int_{0}^{t}\Psi(s,x_{n}(\tau,\cdot),m_{n+1}(\tau,\cdot))d\tau-\int_{0}^{t}\Psi(s,x_{n-1}(\tau,\cdot),m_{n}(\tau,\cdot))d\tau\right|\\
\leq&\left|\int_{0}^{t}\Psi(s,x_{n}(\tau,\cdot),m_{n+1}(\tau,\cdot))-\Psi(s,x_{n}(\tau,\cdot),m_{n}(\tau,\cdot))d\tau\right|\\
&+\left|\int_{0}^{t}\Psi(s,x_{n}(\tau,\cdot),m_{n}(\tau,\cdot))-\Psi(s,x_{n-1}(\tau,\cdot),m_{n}(\tau,\cdot))d\tau\right|\\
\leq&\int_{0}^{t}\left|\Psi(s,x_{n}(\tau,\cdot),m_{n+1}(\tau,\cdot))-\Psi(s,x_{n}(\tau,\cdot),m_{n}(\tau,\cdot))\right|d\tau
\\&+\int_{0}^{t}\left|\Psi(s,x_{n}(\tau,\cdot),m_{n}(\tau,\cdot))-\Psi(s,x_{n-1}(\tau,\cdot),m_{n}(\tau,\cdot))\right|d\tau.
\end{align*}
Integrating in $s\in I$ gives 
\begin{align*}
\int_{I}\left|m_{n+1}(t,s)-m_{n}(t,s)\right|ds\leq&\int_{0}^{t}\int_{I}\left|\Psi(s,x_{n}(\tau,\cdot),m_{n+1}(\tau,\cdot))-\Psi(s,x_{n}(\tau,\cdot),m_{n}(\tau,\cdot))\right|dsd\tau\\
&+\int_{0}^{t}\int_{I}\left|\Psi(s,x_{n}(\tau,\cdot),m_{n}(\tau,\cdot))-\Psi(s,x_{n-1}(\tau,\cdot),m_{n}(\tau,\cdot))\right|dsd\tau.
\end{align*}
Utilizing Lemma \ref{Estimate for psi} shows that the first inner
integral is 
\[
\leq12L\mathbf{S}_{\infty}\sup_{[0,T]\times I}\left|x_{n}\right|\left\Vert m_{n+1}(\tau,\cdot)-m_{n}(\tau,\cdot)\right\Vert _{1}\leq12L\mathbf{S}_{\infty}\overline{X}\left\Vert m_{n+1}(\tau,\cdot)-m_{n}(\tau,\cdot)\right\Vert _{1},
\]
whereas the second inner integral is  
\begin{align*}
\leq L\left(3\overline{M}\mathbf{S}_{\infty}+\mathbf{S}_{\infty}+16\mathbf{S}\overline{X}\right)\sup_{I}\left|x_{1}(t,\cdot)-x_{2}(t,\cdot)\right|.
\end{align*}
As a result, we get 
\begin{align*}
\left\Vert m_{n+1}(t,\cdot)-m_{n}(t,\cdot)\right\Vert _{1}\leq&12L\mathbf{S}_{\infty}\overline{X}\int_{0}^{t}\left\Vert m_{n+1}(\tau,\cdot)-m_{n}(\tau,\cdot)\right\Vert _{1}d\tau\hspace{4.0 cm}\\+&L\left(3\overline{M}\mathbf{S}_{\infty}+\mathbf{S}_{\infty}+16\mathbf{S}\overline{X}\right)\int_{0}^{t}\sup_{I}\left|x_{n}(\tau,\cdot)-x_{n-1}(\tau,\cdot)\right|d\tau \\
\leq& C_{2}(T,X,L,M,\mathbf{S},\mathbf{S}_{\infty})  \left(\right.\int_{0}^{t}\left\Vert m_{n+1}(\tau,\cdot)-m_{n}(\tau,\cdot)\right\Vert _{1}
\end{align*}
\begin{equation}
+\sup_{I}\left|x_{n}(\tau,\cdot)-x_{n-1}(\tau,\cdot)\right|d\tau\left.\right).\hspace{-6.0 cm}
\label{Estimate for mn}
\end{equation}

\noindent \textit{Estimate for $u_{n}(t)\coloneqq\underset{I}{\sup}\left|x_{n+1}(t,\cdot)-x_{n}(t,\cdot)\right|+\left\Vert m_{n+1}(\tau,\cdot)-m_{n}(\tau,\cdot)\right\Vert _{1}$.} Collecting the Inequalities (\ref{ESTIMATE FOR xn}) and (\ref{Estimate for mn})
we find 

\[
u_{n}(t)\leq\overline{C}\left(\int_{0}^{t}u_{n}(\tau)d\tau+\int_{0}^{t}u_{n-1}(\tau)d\tau\right),
\]
where $\overline{C}=C_{1}+C_{2}$. Setting $U_{0}\coloneqq u_{n}(0)$
we find that 
\[
u_{n}(t)\leq e^{\overline{C}T}\int_{0}^{t}u_{n-1}(\tau),
\]
which by easy induction implies 
\[
u_{n}(t)\leq\frac{(e^{\overline{C}T}t)^{n}}{n!}U_{0}.
\]
It follows that 
\[
\underset{[0,T]\times I}{\sup}\left|x_{n+1}-x_{n}\right|+\underset{[0,T]}{\sup}\left\Vert m_{n+1}(\tau,\cdot)-m_{n}(\tau,\cdot)\right\Vert _{1}\underset{n\rightarrow\infty}{\rightarrow}0,
\]
hence $(x_{n},m_{n})$ is a Cauchy sequence in the Banach space $C([0,T]\times I)\oplus C([0,T];L^{1}(I))$
and denote by $(x,m)\in C([0,T]\times I)\oplus C([0,T];L^{\infty}(I))$
the limit point of $(x_{n},m_{n})$. We finish by verifying that $(x,m)$
is a $C^{1}([0,T];C(I))\oplus C^{1}([0,T];L^{\infty}(I))$ solution
to Equation (\ref{eq:COUPLED EQUATION}). Indeed the following equation
is satisfied for each $n$ 

\[
\left\{ \begin{array}{lc}
x_{n}(t,s)=x^{0}(s)+\int_{0}^{t}\int_{I}m_{n-1}(\tau,s_{\ast})\mathbf{a}(x_{n}(\tau,s_{\ast})-x_{n}(\tau,s))ds_{\ast}d\tau\\
m_{n}(t,s)=m^{0}(s)+\int_{0}^{t}\Psi(s,x_{n-1}(\tau,\cdot),m_{n}(\tau,\cdot))d\tau.
\end{array}\right.
\]
We explain how the passage to the limit as $n\rightarrow\infty$ in
the equation for $m_{n}$ is done, and the passage for the equation
of $x_{n}$ is a standard verfication left to the reader. By Claim
\ref{opinions remain seperated } we have for all $(\tau,s_{\ast},s)\in[0,T]\times I^{2}$
\[
\left|x_{n-1}(\tau,s_{\ast})-x_{n-1}(\tau,s)\right|^{2}\geq e^{-LT}\left|x^{0}(s_{\ast})-x^{0}(s)\right|^{2}
\]
so that 

\[
\left|x(\tau,s_{\ast})-x(\tau,s)\right|^{2}\geq e^{-LT}\left|x^{0}(s_{\ast})-x^{0}(s)\right|^{2}.
\]
Therefore, Lemma \ref{Estimate for psi} is applicable and entails
\begin{align*}
\int_{I}\left|\Psi(s,x_{n-1}(\tau,\cdot),m_{n}(\tau,\cdot))-\Psi(s,x(t,\cdot),m(t,\cdot))\right|ds\hspace{-4.0 cm}\\\leq&\int_{I}\left|\Psi(s,x_{n-1}(\tau,\cdot),m_{n}(\tau,\cdot))-\Psi(s,x(t,\cdot),m_{n}(t,\cdot))\right|ds\\
&+\int_{I}\left|\Psi(s,x(\tau,\cdot),m_{n}(\tau,\cdot))-\Psi(s,x(t,\cdot),m(t,\cdot))\right|ds\hspace{-2.0 cm}\\
\lesssim&\left\Vert m_{n}(\tau,\cdot)-m(\tau,\cdot)\right\Vert _{1}+\sup_{I}\left|x_{n-1}(\tau,\cdot)-x(\tau,\cdot)\right|\underset{n\rightarrow\infty}{\rightarrow}0.
\hspace{-4.0cm}
\end{align*}
Hence, it follows that the right hand side in the equation for $m_{n}$
convergence in $C([0,T];L^{1}(I))$ to 
\[
\int_{0}^{t}\Psi(s,x(\tau,\cdot),m(\tau,\cdot))d\tau,
\]
which by uniqueness of the limit implies that for all $t\in[0,T]$
and a.e. $s\in I$ we have 
\[
m(t,s)=m^{0}(s)+\int_{0}^{t}\Psi(s,x(\tau,\cdot),m(\tau,\cdot))d\tau.
\]
The upgrade $(x,m)\in C^{1}([0,T]\times I)\oplus C^{1}([0,T];L^{\infty}(I))$
is exactly by the same reasoning of Lemma \ref{decoupled eq}. 

\textbf{Step 2. Uniqueness}. Suppose we are given 2 solutions $(x_{1},m_{1})$
and $(x_{2},m_{2})$ with the same initial data.We have 
\begin{align*}
\left|x_{1}(t,s)-x_{2}(t,s)\right|=\left|\int_{0}^{t}\int_{I}m_{1}(\tau,s_{\ast})\mathbf{a}(x_{1}(\tau,s_{\ast})-x_{1}(\tau,s))ds_{\ast}d\tau\right.\\
\left.-\int_{0}^{t}\int_{I}m_{2}(\tau,s_{\ast})\mathbf{a}(x_{2}(\tau,s_{\ast})-x_{2}(\tau,s))ds_{\ast}d\tau\right|\\
\leq\left|\int_{0}^{t}\int_{I}m_{1}(\tau,s_{\ast})\mathbf{a}(x_{1}(\tau,s_{\ast})-x_{1}(\tau,s))ds_{\ast}d\tau\right.\\
\left.-\int_{0}^{t}\int_{I}m_{2}(\tau,s_{\ast})\mathbf{a}(x_{1}(\tau,s_{\ast})-x_{1}(\tau,s))ds_{\ast}d\tau\right|\\
+\left|\int_{0}^{t}\int_{I}m_{2}(\tau,s_{\ast})\mathbf{a}(x_{1}(\tau,s_{\ast})-x_{1}(\tau,s))ds_{\ast}d\tau\right.\\
\left.-\int_{0}^{t}\int_{I}m_{2}(\tau,s_{\ast})\mathbf{a}(x_{2}(\tau,s_{\ast})-x_{2}(\tau,s))ds_{\ast}d\tau\right|\\
\lesssim\int_{0}^{t}\left\Vert m_{1}(\tau,\cdot)-m_{2}(\tau,\cdot)\right\Vert _{1}+\underset{I}{\sup}\left|x_{1}(\tau,\cdot)-x_{2}(\tau,\cdot)\right|d\tau. \hspace{-2.0 cm}
\end{align*}
In addition, Lemma \ref{Estimate for psi} yields 
\begin{align*}
 \left\Vert m_{1}(t,\cdot)-m_{2}(t,\cdot)\right\Vert _{1}\leq&\,\int_{0}^{t}\int_{I}\left|\Psi(s,x_{1}(\tau,\cdot),m_{1}(\tau,\cdot))-\Psi(s,x_{2}(\tau,\cdot),m_{2}(\tau,\cdot))\right|dsd\tau
 \\\leq&\, \int_{0}^{t}\int_{I}\left|\Psi(s,x_{1}(t,\cdot),m_{1}(t,\cdot))-\Psi(s,x_{2}(t,\cdot),m_{1}(t,\cdot))\right|dsd\tau\\ &+\int_{0}^{t}\int_{I}\left|\Psi(s,x_{2}(t,\cdot),m_{1}(t,\cdot))-\Psi(s,x_{2}(t,\cdot),m_{2}(t,\cdot))\right|dsd\tau\\
 \lesssim&\,\int_{0}^{t}\left\Vert m_{1}(\tau,\cdot)-m_{2}(\tau,\cdot)\right\Vert _{1}+\sup_{I}\left|x_{1}(\tau,\cdot)-x_{2}(\tau,\cdot)\right|d\tau. \hspace{-2.0 cm}
\end{align*}
It follows that 
\[
\left\Vert m_{1}(t,\cdot)-m_{2}(t,\cdot)\right\Vert _{1}+\sup_{I}\left|x_{1}(t,\cdot)-x_{2}(t,\cdot)\right|\lesssim\int_{0}^{t}\left\Vert m_{1}(\tau,\cdot)-m_{2}(\tau,\cdot)\right\Vert _{1}+\sup_{I}\left|x_{1}(\tau,\cdot)-x_{2}(\tau,\cdot)\right|d\tau,
\]
from which we infer 
\[
\left\Vert m_{1}(t,\cdot)-m_{2}(t,\cdot)\right\Vert _{1}+\sup_{I}\left|x_{1}(t,\cdot)-x_{2}(t,\cdot)\right|=0,
\]
and this a fortiori forces $m_{1}=m_{2}$ and $x_{1}=x_{2}$. 
\begin{flushright}
$\square$
\par\end{flushright}

\section{\label{sec:The-Gronwall-Estimate}The Graph Limit and consequences.}

This section is devoted to obtaining a Gronwall estimate on the time
dependent quantity $\xi_{N}(t)+\zeta_{N}(t)$, where 

\[
\xi_{N}(t)\coloneqq\left\Vert \widetilde{x}_{N}(t,\cdot)-x(t,\cdot)\right\Vert _{L^{2}(I)}^{2},\ \zeta_{N}(t)\coloneqq\left\Vert \widetilde{m}_{N}(t,\cdot)-m(t,\cdot)\right\Vert _{L^{2}(I)}^{2},
\]
where $\widetilde{x}_{N},\widetilde{m}_{N}$ are given by Formula
(\ref{Riemman sums}) and $x,m$ are the corresponding solutions to
Equation (\ref{eq:COUPLED EQUATION}). We modify the argument demonstrated
in Theorem 1 in \cite{1} to our weakly singular settings. The estimate
for $\zeta_{N}(t)$ reflects the main novelty of this section. The
estimates we obtain are locally uniform in time. The symbol $\lesssim$
stands for inequality up to a constant which may depend only on $L,M,X,T,\mathbf{S},\mathbf{S}_{\infty}$.
The main theorem is 
\begin{thm}
\label{the graph limit: main thm} Let the hypotheses \textbf{H1}-\textbf{H3} hold.
Let $(x,m)\in C^{1}\left([0,T];C(I)\right)\oplus C^{1}\left([0,T];L^{\infty}(I)\right)$
be the solution to Equation (\ref{eq:COUPLED EQUATION}). Let $(\mathbf{x}_{N},\mathbf{m}_{N})\in C^{1}\left(\left[0,T\right];\mathbb{R}^{2N}\right)$
be the solution to the system (\ref{opinion dynamics sec2}). Then
\[
\left\Vert \widetilde{x}_{N}(t,\cdot)-x(t,\cdot)\right\Vert _{C([0,T];L^{2}(I))}+\left\Vert \widetilde{m}_{N}(t,\cdot)-m(t,\cdot)\right\Vert _{C([0,T];L^{2}(I))}\underset{N\rightarrow\infty}{\rightarrow}0.
\]
\end{thm}
\textit{Proof.}\textbf{ Step 1. The time derivative of $\zeta_{N}(t)$.}
The estimate for the time derivative of $\zeta_{N}(t)$ reflects the
main difference with the argument in \cite{1}. The time derivative
of $\zeta_{N}$ is computed as follows. 

\begin{align*}
\dot{\zeta_{N}}(t)=\int_{I}\left(\widetilde{m}_{N}(t,s)-m(t,s)\right)\left(N\int_{\frac{1}{N}\left\lfloor sN\right\rfloor }^{\frac{1}{N}(\left\lfloor sN\right\rfloor +1)}\Psi(s_{\ast},\widetilde{x}_{N}(t,\cdot),\widetilde{m}_{N}(t,\cdot))ds_{\ast}-\Psi(s,x(t,\cdot),m(t,\cdot))\right)ds\\
=\int_{I}\left(\widetilde{m}_{N}(t,s)-m(t,s)\right)\underset{\coloneqq h_{N}(t,s)}{\underbrace{\left(N\int_{\frac{1}{N}\left\lfloor sN\right\rfloor }^{\frac{1}{N}(\left\lfloor sN\right\rfloor +1)}\Psi(s_{\ast},\widetilde{x}_{N}(t,\cdot),\widetilde{m}_{N}(t,\cdot))-\Psi(s_{\ast},x(t,\cdot),m(t,\cdot))ds_{\ast}\right)}}ds
\end{align*}
\begin{equation}
+\underset{I}{\int}\left(\widetilde{m}_{N}(t,s)-m(t,s)\right)\underset{\coloneqq g_{N}(t,s)}{\underbrace{\left(N\int_{\frac{1}{N}\left\lfloor sN\right\rfloor }^{\frac{1}{N}(\left\lfloor sN\right\rfloor +1)}\Psi(s_{\ast},x(t,\cdot),m(t,\cdot))ds_{\ast}-\Psi(s,x(t,\cdot),m(t,\cdot))\right)}}ds. 
\label{eq:estimate for =00005Czeta_=00007BN=00007D}
\end{equation}
By Lebesgue differentiation theorem, for each $t\in[0,T]$ it holds
that $g_{N}(t,s)\underset{N\rightarrow\infty}{\rightarrow}0$ pointwise
a.e. $s\in I$. In addition, $\left\Vert x\right\Vert _{C([0,T]\times I)},\left\Vert m\right\Vert _{C([0,T];L^{\infty}(I))}$ are
bounded which implies that $g_{N}(t,s)$ is uniformly bounded (with
respect to $N$) so that by the dominated convergence theorem we find
that the second integral in (\ref{eq:estimate for =00005Czeta_=00007BN=00007D})
is 

\begin{equation}
\leq\frac{1}{2}\zeta_{N}(t)+\frac{1}{2}\left\Vert g_{N}(t,\cdot)\right\Vert _{2}^{2}\label{eq:first estimate for zeta}
\end{equation}
where for each $t\in[0,T]$ it holds that 
\begin{equation}
\left\Vert g_{N}(t,\cdot)\right\Vert _{2}^{2}\underset{N\rightarrow\infty}{\rightarrow}0.\label{eq:VANISHING OF gN}
\end{equation} 
For the first integral, note 
\vspace{-3mm}
\begin{align*}
 \left\Vert h_{N}(t,\cdot)\right\Vert _{2}^{2}=&N^{2}\int_{I}\left|\int_{\frac{1}{N}\left\lfloor sN\right\rfloor }^{\frac{1}{N}(\left\lfloor sN\right\rfloor +1)}\Psi(s_{\ast},\widetilde{x}_{N}(t,\cdot),\widetilde{m}_{N}(t,\cdot))-\Psi(s_{\ast},x(t,\cdot),m(t,\cdot))ds_{\ast}\right|^{2}ds\\
 =&N^{2}\stackrel[i=1]{N}{\sum}\int_{\frac{i-1}{N}}^{\frac{i}{N}}\left|\int_{\frac{1}{N}\left\lfloor sN\right\rfloor }^{\frac{1}{N}(\left\lfloor sN\right\rfloor +1)}\Psi(s_{\ast},\widetilde{x}_{N}(t,\cdot),\widetilde{m}_{N}(t,\cdot))-\Psi(s_{\ast},x(t,\cdot),m(t,\cdot))ds_{\ast}\right|^{2}ds
\\ \leq& N\stackrel[i=1]{N}{\sum}\int_{\frac{i-1}{N}}^{\frac{i}{N}}\int_{\frac{1}{N}\left\lfloor sN\right\rfloor }^{\frac{1}{N}(\left\lfloor sN\right\rfloor +1)}\left|\Psi(s_{\ast},\widetilde{x}_{N}(t,\cdot),\widetilde{m}_{N}(t,\cdot))-\Psi(s_{\ast},x(t,\cdot),m(t,\cdot))\right|^{2}ds_{\ast}ds\\
=&N\stackrel[i=1]{N}{\sum}\int_{\frac{i-1}{N}}^{\frac{i}{N}}\int_{\frac{i-1}{N}}^{\frac{i}{N}}\left|\Psi(s_{\ast},\widetilde{x}_{N}(t,\cdot),\widetilde{m}_{N}(t,\cdot))-\Psi(s_{\ast},x(t,\cdot),m(t,\cdot))\right|^{2}ds_{\ast}ds\\
=&\stackrel[i=1]{N}{\sum}\int_{\frac{i-1}{N}}^{\frac{i}{N}}\left|\Psi(s_{\ast},\widetilde{x}_{N}(t,\cdot),\widetilde{m}_{N}(t,\cdot))-\Psi(s_{\ast},x(t,\cdot),m(t,\cdot))\right|^{2}ds_{\ast}\\
=&\int_{I}\left|\Psi(s_{\ast},\widetilde{x}_{N}(t,\cdot),\widetilde{m}_{N}(t,\cdot))-\Psi(s_{\ast},x(t,\cdot),m(t,\cdot))\right|^{2}ds_{\ast}\\ 
\leq&2\int_{I}\left|\Psi(s_{\ast},\widetilde{x}_{N}(t,\cdot),\widetilde{m}_{N}(t,\cdot))-\Psi(s_{\ast},\widetilde{x}_{N}(t,\cdot),m(t,\cdot))\right|^{2}ds_{\ast}\\
&+2\int_{I}\left|\Psi(s_{\ast},\widetilde{x}_{N}(t,\cdot),m(t,\cdot))-\Psi(s_{\ast},x(t,\cdot),m(t,\cdot))\right|^{2}ds_{\ast}\\
=&2\int_{I}\left|\Psi(s,\widetilde{x}_{N}(t,\cdot),\widetilde{m}_{N}(t,\cdot))-\Psi(s,\widetilde{x}_{N}(t,\cdot),m(t,\cdot))\right|^{2}ds\\
&+2\int_{I}\left|\Psi(s,\widetilde{x}_{N}(t,\cdot),m(t,\cdot))-\Psi(s,x(t,\cdot),m(t,\cdot))\right|^{2}ds.
\end{align*}
Note that at this stage we cannot quite appeal to the Estimate (\ref{Estimate for psi})
since it was formulated for $x$ which are one to one in the variable
$s$ . The main difference is in the estimate of the second integral,
which is now bounded by $\left\Vert \widetilde{x}_{N}(t,\cdot)-x(t,\cdot)\right\Vert _{2}^{2}$
up to an error term which decays to $0$ as $N\rightarrow\infty$.
Precisely put
\begin{lem}
    
It holds that 
\[
i.\int_{I}\left|\Psi(s,\widetilde{x}_{N}(t,\cdot),\widetilde{m}_{N}(t,\cdot))-\Psi(s,\widetilde{x}_{N}(t,\cdot),m(t,\cdot))\right|^{2}ds\lesssim\left\Vert \widetilde{m}_{N}(t,\cdot)-m(t,\cdot)\right\Vert _{2}^{2}.
\]

\[
ii.\int_{I}\left|\Psi(s,\widetilde{x}_{N}(t,\cdot),m(t,\cdot))-\Psi(s,x(t,\cdot),m(t,\cdot))\right|^{2}ds\lesssim\left\Vert \widetilde{x}_{N}(t,\cdot)-x(t,\cdot)\right\Vert _{2}^{2}+o_{N}(1).
\]
\end{lem}
\textit{Proof}. Unless unavoidable, we supress the time variable.

\textbf{i.} Thanks to Lemma \ref{basic properties } we have 
\[
\int_{I}\widetilde{m}_{N}(t,s)ds=1,\ \widetilde{m}_{N}(t,s)\geq0.
\]
The estimate i. is almost identical to the estimate demonstrated in
(\ref{Estimate for psi}), the only minor difference being that here
we take the $L^{2}$-norm in $s$. Setting $\mathbf{a}_{N}(s,s_{\ast},s_{\ast\ast})\coloneqq\mathbf{a}(\widetilde{x}_{N}(s_{\ast\ast})-\widetilde{x}_{N}(s))+\mathbf{a}(\widetilde{x}_{N}(s_{\ast\ast})-\widetilde{x}_{N}(s_{\ast}))$, we have 
\begin{align*}
 \left|\widetilde{m}_{N}(s)\iint_{I^{2}}\widetilde{m}_{N}(s_{\ast})\widetilde{m}_{N}(s_{\ast\ast})\mathbf{a}_{N}(s,s_{\ast},s_{\ast\ast})\mathbf{s}(\widetilde{x}_{N}(s)-\widetilde{x}_{N}(s_{\ast}))ds_{\ast}ds_{\ast\ast}\right.\hspace{4.0 cm}\\
 \left.-m(s)\iint_{I^{2}}m(s_{\ast})m(s_{\ast\ast})\mathbf{a}_{N}(s,s_{\ast},s_{\ast\ast})\mathbf{s}(\widetilde{x}_{N}(s)-\widetilde{x}_{N}(s_{\ast}))ds_{\ast}ds_{\ast\ast}\right|\\
 \leq\widetilde{m}_{N}(s)\left|\iint_{I^{2}}\left(\widetilde{m}_{N}(s_{\ast})\widetilde{m}_{N}(s_{\ast\ast})-m(s_{\ast})m(s_{\ast\ast})\right)\mathbf{a}(\widetilde{x}_{N}(s_{\ast\ast})-\widetilde{x}_{N}(s))\mathbf{s}(\widetilde{x}_{N}(s)-\widetilde{x}_{N}(s_{\ast}))ds_{\ast}ds_{\ast\ast}\right. \hspace{-0.5 cm}\\\left.+\iint_{I^{2}}\left(\widetilde{m}_{N}(s_{\ast})\widetilde{m}_{N}(s_{\ast\ast})-m(s_{\ast})m(s_{\ast\ast})\right)\mathbf{a}(\widetilde{x}_{N}(s_{\ast\ast})-\widetilde{x}_{N}(s_{\ast}))\mathbf{s}(\widetilde{x}_{N}(s)-\widetilde{x}_{N}(s_{\ast}))ds_{\ast}ds_{\ast\ast}\right|\hspace{-0.5 cm}\\
 +\left|\widetilde{m}_{N}(s)-m(s)\right|\left|\iint_{I^{2}}m(s_{\ast})m(s_{\ast\ast})\mathbf{a}(\widetilde{x}_{N}(s_{\ast\ast})-\widetilde{x}_{N}(s))\mathbf{s}(\widetilde{x}_{N}(s)-\widetilde{x}_{N}(s_{\ast}))ds_{\ast}ds_{\ast\ast}\right.\\\left.+\iint_{I^{2}}m(s_{\ast})m(s_{\ast\ast})\mathbf{a}(\widetilde{x}_{N}(s_{\ast\ast})-\widetilde{x}_{N}(s_{\ast}))\mathbf{s}(\widetilde{x}_{N}(s)-\widetilde{x}_{N}(s_{\ast}))ds_{\ast}ds_{\ast\ast}\right|\\
 \lesssim\iint_{I^{2}}\left|\widetilde{m}_{N}(s_{\ast})\widetilde{m}_{N}(s_{\ast\ast})-m(s_{\ast})m(s_{\ast\ast})\right|ds_{\ast}ds_{\ast\ast}
 \hspace{3.0 cm}
 \end{align*}
\begin{equation}
+\left|\widetilde{m}_{N}(s)-m(s)\right|\iint_{I^{2}}m(s_{\ast})m(s_{\ast\ast})ds_{\ast}ds_{\ast\ast}. \hspace{-3.0 cm}\label{eq:-1-1}
\end{equation}
The first integral is 
\begin{align*}
\lesssim&\iint_{I^{2}}\left|\widetilde{m}_{N}(s_{\ast\ast})\left(\widetilde{m}_{N}(s_{\ast})-m(s_{\ast})\right)\right|ds_{\ast}ds_{\ast\ast}+\iint_{I^{2}}\left|m(s_{\ast})\left(\widetilde{m}_{N}(s_{\ast\ast})-m(s_{\ast\ast})\right)\right|ds_{\ast}ds_{\ast\ast}\\
\lesssim&\int_{I}\left|\widetilde{m}_{N}(s_{\ast\ast})-m(s_{\ast\ast})\right|ds_{\ast\ast}+\int_{I}\left|\widetilde{m}_{N}(s_{\ast})-m(s_{\ast})\right|ds_{\ast}.
\end{align*}
Therefore squaring and integrating in $s$ over $I$, Inequality (\ref{eq:-1-1})
produces 
\[
\int_{I}\left|\Psi(s,\widetilde{x}_{N}(t,\cdot),\widetilde{m}_{N}(t,\cdot))-\Psi(s,\widetilde{x}_{N}(t,\cdot),m(t,\cdot))\right|^{2}ds\lesssim\left\Vert \widetilde{m}_{N}(t,\cdot)-m(t,\cdot)\right\Vert _{2}^{2}.
\]

\textbf{ii. }We have 

\begin{align*}
&\left|\iint_{I^{2}}m(s_{\ast})m(s_{\ast\ast})\mathbf{a}_{N}(s,s_{\ast},s_{\ast\ast})\mathbf{s}(\widetilde{x}_{N}(s)-\widetilde{x}_{N}(s_{\ast}))ds_{\ast}ds_{\ast\ast}\right.\\    
&\left.-\iint_{I^{2}}m(s_{\ast})m(s_{\ast\ast})\mathbf{a}(s,s_{\ast},s_{\ast\ast})\mathbf{s}(x(s)-x(s_{\ast}))ds_{\ast}ds_{\ast\ast}\right|\\
\leq& L\iint_{I^{2}}m(s_{\ast})m(s_{\ast\ast})\left(2\left|\widetilde{x}_{N}(s_{\ast\ast})-x(s_{\ast\ast})\right|+\left|\widetilde{x}_{N}(s_{\ast})-x(s_{\ast})\right|+\left|\widetilde{x}_{N}(s)-x(s)\right|\right)ds_{\ast}ds_{\ast\ast}\\
&+\iint_{I^{2}}m(s_{\ast})m(s_{\ast\ast})\left|\mathbf{a}(x(s_{\ast\ast})-x(s))+\mathbf{a}(x(s_{\ast\ast})-x(s_{\ast}))\right|\\
&\times\left|\mathbf{s}(\widetilde{x}_{N}(s)-\widetilde{x}_{N}(s_{\ast}))-\mathbf{s}(x(s)-x(s_{\ast}))\right|ds_{\ast}ds_{\ast\ast}\\&\coloneqq J_{1}(t,s)+J_{2}(t,s).
\end{align*}
We estimate separately $\int_{I}m^{2}(t,s)\left|J_{1}(t,s)\right|^{2}ds$ and
$\int_{I}m^{2}(t,s)\left|J_{2}(t,s)\right|^{2}ds$. It is straightforward to
check
\begin{equation}
\int_{I}m^{2}(t,s)\left|J_{1}(t,s)\right|^{2}ds\lesssim\left\Vert \widetilde{x}_{N}(t,\cdot)-x(t,\cdot)\right\Vert _{2}^{2}.\label{eq:J1 es}
\end{equation}
Let us concentrate on the estimate of $\int_{I}m^{2}(t,s)\left|J_{2}(t,s)\right|^{2}ds$.
For each $s\in I$ set
\[
A_{N}(t,s)\coloneqq\left\{ s_{\ast}\in I\left|\widetilde{x}_{N}(t,s_{\ast})-\widetilde{x}_{N}(t,s)>0\right.\right\}, 
\ B_{N}(t,s)\coloneqq A_{N}^{c}(t,s)\]
and 
\[
A(t,s)\coloneqq\left\{ s_{\ast}\in I\left|x(t,s_{\ast})-x(t,s)>0\right.\right\}, \ B(t,s)\coloneqq A^{c}(t,s) .
\]
We abbreviate 
\[\mathbf{s}_{N}(s,s_{\ast})\coloneqq\mathbf{s}(\widetilde{x}_{N}(t,s)-\widetilde{x}_{N}(t,s_{\ast}))-\mathbf{s}(x(t,s)-x(t,s_{\ast})).\]
We estimate the integral as follows 
\begin{align*}
 \int_{I}m^{2}(t,s)\left|J_{2}(t,s)\right|^{2}ds\lesssim&\iint_{I^{2}}\left|\mathbf{s}_{N}(s,s_{\ast})\right|^{2}ds_{\ast}ds\\
=&\iint_{I^{2}}\mathbf{1}_{A_{N}(t,s)}(s_{\ast})\mathbf{1}_{A(t,s)}(s_{\ast})\left|\mathbf{s}_{N}(s,s_{\ast})\right|^{2}ds_{\ast}ds\\
&+\iint_{I^{2}}\mathbf{1}_{B_{N}(t,s)}(s_{\ast})\mathbf{1}_{B(t,s)}(s_{\ast})\left|\mathbf{s}_{N}(s,s_{\ast})\right|^{2}ds_{\ast}ds\\
&+\iint_{I^{2}}\mathbf{1}_{A_{N}(t,s)}(s_{\ast})\left(\mathbf{1}_{B(t,s)}-\mathbf{1}_{B_{N}(t,s)}\right)(s_{\ast})\left|\mathbf{s}_{N}(s,s_{\ast})\right|^{2}ds_{\ast}ds\\
&+\iint_{I^{2}}\left(\mathbf{1}_{A(t,s)}-\mathbf{1}_{A_{N}(t,s)}\right)(s_{\ast})\mathbf{1}_{B_{N}(t,s)}(s_{\ast})\left|\mathbf{s}_{N}(s,s_{\ast})\right|^{2}ds_{\ast}ds\\\eqqcolon&\mathbf{I}+\mathbf{II}+\mathbf{III}+\mathbf{IV}.
\end{align*}
By the assumption \textbf{H3} we have
\begin{equation}
\mathbf{I}+\mathbf{II}\lesssim\left\Vert \widetilde{x}_{N}(t,\cdot)-x(t,\cdot)\right\Vert _{2}^{2}=\xi_{N}(t).\label{I+II}    
\end{equation}
Recall that by Lemma \ref{basic properties } and Claim \ref{opinions remain seperated } there exist a constant
$C>1$ such that 

\begin{equation}
\frac{1}{C}\left|\widetilde{x}_{N}(0,s_{\ast})-\widetilde{x}_{N}(0,s)\right|\leq\left|\widetilde{x}_{N}(t,s_{\ast})-\widetilde{x}_{N}(t,s)\right|\leq C\left|\widetilde{x}_{N}(0,s_{\ast})-\widetilde{x}_{N}(0,s)\right|,\label{eq:-6}
\end{equation}

\[
\frac{1}{C}\left|x^{0}(s_{\ast})-x^{0}(s)\right|\leq\left|x(t,s_{\ast})-x(t,s)\right|\leq C\left|x^{0}(s_{\ast})-x^{0}(s)\right|.
\]
Hence $\mathbf{1}_{A_{N}(t,s)}=\mathbf{1}_{A_{N}(0,s)},\mathbf{1}_{A(t,s)}=\mathbf{1}_{A(0,s)}$,
so that 

\[
\mathbf{IV}\lesssim\iint_{I^{2}}\left|\mathbf{1}_{A_{N}(t,s)}(s_{\ast})-\mathbf{1}_{A(t,s)}(s_{\ast})\right|ds_{\ast}ds=\iint_{I^{2}}\left|\mathbf{1}_{A_{N}(0,s)}(s_{\ast})-\mathbf{1}_{A(0,s)}(s_{\ast})\right|ds_{\ast}ds.
\]
By Lebesgue differentiation theorem, for a.e. $s\in I$ it holds that 
\[\widetilde{x_{N}}(0,s_{\ast})-\widetilde{x_{N}}(0,s)\underset{N\rightarrow\infty}{\rightarrow}x^{0}(s_{\ast})-x^{0}(s)\]
a.e. $s_{\ast}$. 
For all $s\in I$ the set $\left\{ s_{\ast}\left|x^{0}(s_{\ast})=x^{0}(s)\right.\right\} $(being
an atom) is null due to \textbf{H2}, and therefore for a.e. $s$ it holds that 
\[
\mathbf{\mathbf{1}}_{A_{N}(0,s)}(s_{\ast})\underset{N\rightarrow\infty}{\rightarrow}\mathbf{\mathbf{1}}_{A(0,s)}(s_{\ast})
\] 
 a.e. $ s_{\ast}$. 
By dominated convergence we obtain 

\[
\iint_{I^{2}}\left|\mathbf{1}_{A_{N}(0,s)}-\mathbf{1}_{A(0,s)}\right|ds_{\ast}ds\underset{N\rightarrow\infty}{\rightarrow}0,
\]
which shows 
\begin{equation}
\mathbf{IV}\lesssim\iint_{I^{2}}\left|\mathbf{1}_{A_{N}(t,s)}-\mathbf{1}_{A(t,s)}\right|ds_{\ast}ds=o_{N}(1).\label{eq:oN(1)}
\end{equation}
The same reasoning also shows that 
\begin{equation}
\mathbf{III}\lesssim\iint_{I^{2}}\left|\mathbf{1}_{B_{N}(t,s)}-\mathbf{1}_{B(t,s)}\right|ds_{\ast}ds=o_{N}(1).\label{eq:oN(1)B}
\end{equation}
The combination of (\ref{eq:J1 es}), (\ref{I+II}), (\ref{eq:oN(1)}) and (\ref{eq:oN(1)B})
implies the announced claim. 
\begin{flushright}
$\square$
\par\end{flushright}
Gathering i.,ii. and (\ref{eq:first estimate for zeta}) gives
\begin{equation}
\dot{\zeta_{N}}(t)\lesssim\xi_{N}(t)+\zeta_{N}(t)+\left\Vert g_{N}(t,\cdot)\right\Vert _{2}^{2}+o_{N}(1).\label{eq:}
\end{equation}

\textbf{Step 2. The time derivative of $\xi_{N}(t)$.} The time derivative
of $\dot{\xi_{N}}(t)$ is mastered exactly as in {[}1{]}. Following
the argument in {[}1{]} one finds that 
\begin{equation}
\dot{\xi}_{N}(t)\lesssim\xi_{N}(t)+\zeta_{N}(t).\label{eq:-3}
\end{equation}

\textbf{Step 3. Conclusion. }The combination of Inequalities (\ref{eq:})
and (\ref{eq:-3}) yields 

\begin{align*}
  \dot{\xi}_{N}(t)+\dot{\zeta}_{N}(t)\lesssim\xi_{N}(t)+\zeta_{N}(t)+\frac{1}{2}\left\Vert g_{N}(t,\cdot)\right\Vert _{2}^{2}+o_{N}(1),  
\end{align*}
i.e. 
\begin{align*}
\xi_{N}(t)+\zeta_{N}(t)\lesssim\xi_{N}(0)+\zeta_{N}(0)+\int_{0}^{t}(\xi_{N}(\tau)+\zeta_{N}(\tau))d\tau+\int_{0}^{T}\left\Vert g_{N}(\tau,\cdot)\right\Vert _{2}^{2}d\tau+o_{N}(1)T.    
\end{align*}
Applying Gronwall's lemma entails 

\[
\xi_{N}(t)+\zeta_{N}(t)\leq C\left(\xi_{N}(0)+\zeta_{N}(0)+\int_{0}^{T}\left\Vert g_{N}(\tau,\cdot)\right\Vert _{2}^{2}d\tau+o_{N}(1)T\right)\exp\left(Ct\right).
\]
Since $\left\Vert g_{N}(\tau,\cdot)\right\Vert _{2}^{2}$ is uniformly
bounded, by (\ref{eq:VANISHING OF gN}) and dominated convergence
$\int_{0}^{T}\left\Vert g_{N}(\tau,\cdot)\right\Vert _{2}^{2}d\tau\underset{N\rightarrow\infty}{\rightarrow}0$,
which concludes the proof. 
\begin{flushright}
$\square$
\par\end{flushright}
In the last part of this section, we recall how to obtain as a consequence a special version of the mean
field limit for the empirical measure associated with the System (\ref{opinion dynamics sec2}).
We start by pointing out that currently the existing literature does
not cover the well posedness theory of the mean field equation, namely
the non-local non-homogeneous transport equation 
\[
\partial_{t}\mu(t,x)+\partial_{x}\left(\mu(t,x)\mathbf{a}\star\mu(t,x)\right)=h\left[\mu\right](t,x),\ \mu(0,\cdot)=\mu^{0},
\]
where 
\[
h\left[\mu\right](t,x)\coloneqq d\mu(t,x)\int_{\mathbb{R}^{2}}S(x,y,z)d\mu(t,y)d\mu(t,z),\ S(x,y,z)=\frac{1}{2}\left(\mathbf{a}(z-y)+\mathbf{a}(z-x)\right)\mathbf{s}(x-y).
\]
Indeed, in both works \cite{5,15} the well posedness
theory for measure valued solutions made an extensive use of the fact
that the source term satisfies some kind of Lipschitz continuity with
respect to a generalized Wasserstein distance $\rho$ defined in \cite{5,15}, namely a
bound of the type $\rho(h\left[\mu\right],h\left[\nu\right])\leq C\rho(\mu,\nu)$
with $C>0$. In fact, the graph limit method already allows to get a clean characterization of the
mean field limit for the special choice of initial data (\ref{initial data}).
With Theorem \ref{the graph limit: main thm} the proof of this mean
field limit can be deduced exactly by the same argument as in \cite{1}.
The following theorem is the first step in
understanding the mean field limit of the pairwise competition model,
which was a question raised in \cite{10}. We denote by $W_{1}$ the
Wasserstein distance of exponent $1$ (see e.g. \cite[Definition
6.1]{16} for general background on Wasserstein distances). 

\begin{thm}
Let the hypothesis \textbf{H1}-\textbf{H3} hold. Let $\mu_{N}(t)\coloneqq\frac{1}{N}\stackrel[j=1]{N}{\sum}m_{j}(t)\delta(x-x_{j}(t))$
and $\mu(t)\coloneqq\int_{I}m(t,s)\delta(x-x(t,s))ds$\footnote{We view $\int_{I}m(t,s)\delta(x-x(t,s))ds$ as a measure valued integral.
By definition $\left(\int_{I}m(t,s)\delta(x-x(t,s))ds\right)\varphi\coloneqq\int_{I}m(t,s)\varphi(x(t,s))ds$
for all $\varphi\in C_{b}(\mathbb{R}^{d})$.}. Then $W_{1}(\mu_{N}(t),\mu(t))\underset{N\rightarrow\infty}{\rightarrow}0$
for all $t\in[0,T]$. \label{weak mean field}
\end{thm}
\textit{Proof}. Fix some $\varphi\in C^{0,1}(\mathbb{R}^{d})$. Let
\[
\overline{\mu_{N}}(t)\coloneqq\int_{I}\widetilde{m_{N}}(t,s)\delta(x-\widetilde{x_{N}}(t,s))ds.
\]
We split the integral $\int_{\mathbb{R}}\varphi(x)\left(\mu_{N}(t,dx)-\mu(t,dx)\right)$
as follows. 

\begin{equation}
\int_{\mathbb{R}}\varphi(x)\left(\mu_{N}(t,dx)-\mu(t,dx)\right)=\int_{\mathbb{R}}\varphi(x)\left(\mu_{N}(t,dx)-\overline{\mu_{N}}(t,dx)\right)+\int_{\mathbb{R}}\varphi(x)\left(\overline{\mu_{N}}(t,dx)-\mu(t,dx)\right).\label{eq:-4}
\end{equation}
The first integral in the right hand side of Equation (\ref{eq:-4})
vanishes identically. Indeed 
\begin{align*}
 \int_{\mathbb{R}}\varphi(x)\left(\mu_{N}(t,dx)-\overline{\mu_{N}}(t,dx)\right)=&\frac{1}{N}\stackrel[j=1]{N}{\sum}m_{j}(t)\varphi(x_{j}(t))\\&-\stackrel[i=1]{N}{\sum}\int_{I}\varphi\left(\stackrel[j=1]{N}{\sum}\mathbf{1}_{\left[\frac{j-1}{N},\frac{j}{N}\right]}(s)x_{j}(t)\right)m_{i}(t)\mathbf{1}_{\left[\frac{i-1}{N},\frac{i}{N}\right]}(s)ds,   
\end{align*}
and the second term is
\begin{align*}
=&\stackrel[i=1]{N}{\sum}\int_{I}\varphi\left(\stackrel[j=1]{N}{\sum}\mathbf{1}_{\left[\frac{j-1}{N},\frac{j}{N}\right]}(s)x_{j}(t)\right)m_{i}(t)\mathbf{1}_{\left[\frac{i-1}{N},\frac{i}{N}\right]}(s)ds\\
=&\stackrel[i=1]{N}{\sum}\int_{\left[\frac{i-1}{N},\frac{i}{N}\right]}\varphi\left(\stackrel[j=1]{N}{\sum}\mathbf{1}_{\left[\frac{j-1}{N},\frac{j}{N}\right]}(s)x_{j}(t)\right)m_{i}(t)ds\\
=&\stackrel[i=1]{N}{\sum}\int_{\left[\frac{i-1}{N},\frac{i}{N}\right]}\varphi\left(x_{i}(t)\right)m_{i}(t)ds\\=&\frac{1}{N}\stackrel[i=1]{N}{\sum}m_{i}(t)\varphi\left(x_{i}(t)\right).
\end{align*}
As for the second integral, it is
\begin{align*}
 =&\int_{\mathbb{R}}\varphi(x)\left(\int_{I}\widetilde{m_{N}}(t,s)\delta(x-\widetilde{x_{N}}(t,s))ds-\int_{I}m(t,s)\delta(x-x(t,s))ds\right)\\
 =&\int_{I}\widetilde{m_{N}}(t,s)\varphi(\widetilde{x_{N}}(t,s))ds-\int_{I}m(t,s)\varphi(x(t,s))ds\\\leq&\left|\int_{I}\left(\widetilde{m_{N}}(t,s)-m(t,s)\right)\varphi(\widetilde{x_{N}}(t,s))ds\right|\\&+\left|\int_{I}m(t,s)\left(\varphi(\widetilde{x_{N}}(t,s))-\varphi(x(t,s))\right)ds\right|
 \\\leq&\left\Vert \varphi\right\Vert _{\infty}\left\Vert \widetilde{m}_{N}(t,\cdot)-m(t,\cdot)\right\Vert _{L^{2}(I)}\\&+\overline{M}\left\Vert \varphi\right\Vert _{\mathrm{Lip}}\left\Vert \widetilde{x}_{N}(t,\cdot)-x(t,\cdot)\right\Vert _{L^{2}(I)}.
 \end{align*}
The last estimate together with Theorem \ref{the graph limit: main thm}
entails the weak convergnce $\mu_{N}(t)\rightharpoonup\mu(t)$. Since
$\mu_{N}(t)$ and $\mu(t)$ are compactly supported, it follows that
$W_{1}(\mu_{N}(t),\mu(t))\underset{N\rightarrow\infty}{\rightarrow}0$
(see e.g. \cite[Theorem 6.9]{16}). 
\begin{flushright}
$\square$
\par\end{flushright}
\section{The case $d>1$}\
\label{Section 5}
In this section we explain how to extend the graph limit for arbitrary higher dimensions $d>1$. In some places the proof requires only minor modifications and we therefore concentrate only on the parts which require special treatment.

\subsection{The graph limit equation $d>1$}
The first notable difference in comparison to the case $d=1$ (or the work \cite{1}) is reflected in the definition of the Riemann sums. Instead of labeling the opinions along a multi-index of length $d$ we label them along a multi-$d$-dimensional matrix of indices. This is a particular case of the metric valued labelling procedure introduced in \cite{12} when the labelling space is $[0,1]^d$. At the level of the graph limit equation this choice corresponds to considering the equation posed on $[0,T]\times I^{d}$ rather than $[0,T]\times I$. Indeed, the fact that $x(t,\cdot)$ is a map from $I^{d}$ to itself enables to consider  bi-Lipschitz initial data, which is crucial for the sake of properly analyzing the singularity in $\mathbf{s}$ as is clarified in Lemma \ref{Estimate for psi d>1}. This labelling procedure does not have any modelling interpretation since particles (opinions) are still exchangeable or indistinguishable, it is solely needed for pure technical reasons. As we mentioned at the beginning of the paper, it would still be possible to go back to using $[0, 1]$ as a labeling space through a change of variable, since $[0, 1]$ and $[0,\ 1]^d$ are isomorphic as measurable spaces per the Borel isomorphism theorem. But, obviously, this would lead to painful technical assumptions to replace the bi-Lipschitz condition on $x(t,.)$, while the analysis is otherwise much more transparent when considering $[0,\ 1]^d$.   Precisely put, we take the number of opinions to be perfect powers of $d$ in which case the opinion dynamics system becomes 
the following system of $(d+1)N^d$ ODEs 
\begin{equation}
\left\{ \begin{array}{lc}
\dot{x_{i}}^{N}(t)=\frac{1}{N^d}\stackrel[j=1]{N^d}{\sum}m_{j}^{N}(t)\mathbf{a}(x_{j}^N(t)-x_{i}^N(t)),&\ x_{i}^{N}(0)=x_{i}^{0,N}\\
\dot{m}_{i}^{N}(t)=\psi_{i}^{N}(\mathbf{x}_{N}(t),\mathbf{m}_{N}(t)),&\ m_{i}^{N}(0)=m_{i}^{0,N}.
\end{array}\right.\label{eq:OPINION DYNAMICS d>1}
\end{equation}
Let $\mathcal{\mathcal{M}}=\left\{ 1,\ldots,N\right\}^d$ be the set of sequences of length $d$
with elements from $\left\{ 1,\ldots,N\right\}$ and fix a bijection
$\sigma:\mathcal{\mathcal{M}}\rightarrow\left\{ 1,\ldots,N^{d}\right\} $.
For each $\mathbf{i}=\left(i_{1},\ldots,i_{d}\right)\in\mathcal{M}$
consider the cube
\[
\mathcal{Q}_{\mathbf{i}}\coloneqq\stackrel[k=1]{d}{\prod}\left[\frac{i_{k}-1}{N},\frac{i_{k}}{N}\right].
\]
We attach to the flow of System (\ref{eq:OPINION DYNAMICS d>1})
the following ``Riemman sums'' like quantities, as in the one dimensional case, defined by
\begin{equation}
\widetilde{x}_{N}(t,s)\coloneqq\underset{\mathbf{i}\in\mathcal{\mathcal{M}}}{\sum}x_{\sigma\left(\mathbf{i}\right)}(t)\mathbf{1}_{\mathcal{Q}_{\left(\mathbf{i}\right)}}(s)=\stackrel[i_{d}=1]{N}{\sum}\cdots\stackrel[i_{1}=1]{N}{\sum}x_{\sigma\left(i_{1},\ldots,i_{d}\right)}(t)\mathbf{1}_{\mathcal{Q}_{\left(i_{1},\ldots,i_{d}\right)}}(s),\label{Riemman sums}
\end{equation}

\begin{equation}
\widetilde{m}_{N}(t,s)\coloneqq\underset{\mathbf{i}\in\mathcal{\mathcal{M}}}{\sum}m_{\sigma\left(\mathbf{i}\right)}(t)\mathbf{1}_{\mathcal{Q}_{\left(\mathbf{i}\right)}}(s)=\stackrel[i_{d}=1]{N}{\sum}\cdots\stackrel[i_{1}=1]{N}{\sum}m_{\sigma\left(i_{1},\ldots,i_{d}\right)}(t)\mathbf{1}_{\mathcal{Q}_{\left(i_{1},\ldots,i_{d}\right)}}(s).\label{riemman sums weights}
\end{equation}
Here the labeling variable $s$ varies on the $d$-dimensional unit cube $I^d$. Generalizing the constructions of Section \ref{graph limit eq 1D}, the functional $\Psi:I^{d}\times L^{\infty}(I^{d})\times L^{\infty}(I^{d})\rightarrow\mathbb{R}$
and the functions $x^{0}:I^{d}\rightarrow\mathbb{R}^{d},m^{0}:I^{d}\rightarrow\mathbb{R}$
are given and the functions $\psi_{i}^{N}$ and the initial data
$x_{i}^{0,N},m_{i}^{0,N}$ ($1\leq i\leq N^{d}$) are defined in terms
of these functions through the following formulas

\begin{equation}
\psi_{i}^{N}(\mathbf{x}_{N}(t),\mathbf{m}_{N}(t))\coloneqq N^{d}\int_{\mathcal{Q}_{\sigma^{-1}(i)}}\Psi(s,\widetilde{x}_{N}(t,\cdot),\widetilde{m}_{N}(t,\cdot))ds\label{eq:formula for psi sec 5}
\end{equation}
and 

\begin{equation}
x_{i}^{0,N}\coloneqq N^{d}\int_{\mathcal{Q}_{\sigma^{-1}(i)}}x^{0}(s)ds,\ m_{i}^{0,N}=N^{d}\int_{\mathcal{Q}_{\sigma^{-1}(i)}}m^{0}(s)ds.\label{initial data sec 5}
\end{equation}
If $\Psi$ is given by 

\[
\Psi(s,x(\cdot),m(\cdot))\coloneqq m(s)\iint_{I^{2d}}m(s_{\ast})m(s_{\ast\ast})\left(\mathbf{a}(x(s_{\ast\ast})-x(s))+\mathbf{a}(x(s_{\ast\ast})-x(s_{\ast}))\right)\mathbf{s}(x(s)-x(s_{\ast}))ds_{\ast}ds_{\ast\ast},
\]
then one readily checks that the $\psi_{i}^{N}$ in Formula (\ref{eq:formula for psi})
are recovered via Formula (\ref{eq:formula for psi sec 5}). Notice that $\widetilde{x}_{N}(0,s)$ and $\widetilde{m}_{N}(0,s)$ well approximate $x^{0}(s),m^{0}(s)$ because by Lebesgue's
differentiation theorem for a.e. $s\in I^{d}$ we have pointwise convergence 

\[
\widetilde{x}_{N}(0,s)=N^d\int_{\mathcal{Q}_{\left\lfloor s_{1}N\right\rfloor ,\ldots,\left\lfloor s_{d}N\right\rfloor }}x^{0}(\sigma)d\sigma\underset{N\rightarrow\infty}{\rightarrow}x^{0}(s)
\]
and 

\[
\widetilde{m}_{N}(0,s)=N^d\int_{\mathcal{Q}_{\left\lfloor s_{1}N\right\rfloor ,\ldots,\left\lfloor s_{d}N\right\rfloor }}m^{0}(\sigma)d\sigma\underset{N\rightarrow\infty}{\rightarrow}m^{0}(s).
\]
The functions $\widetilde{x}_{N}(t,s),\widetilde{m}_{N}(t,s)$ defined
through formulas (\ref{Riemman sums}), (\ref{riemman sums weights})
are governed by the following equation, which is the obvious higher dimensional version of Equation (\ref{eq:integro differential system intro}).
\begin{prop}
\label{equation governing xN,mN d>1} Let the assumptions of Proposition
\ref{well posedness of opinion dynamics } hold and let $(\mathbf{x}_{N}(t),\mathbf{m}_{N}(t))$
be the solution to System (\ref{eq:OPINION DYNAMICS d>1}) on $[0,T]$.
Let $\widetilde{x}_{N},\widetilde{m}_{N}$ be given by (\ref{Riemman sums}) and (\ref{riemman sums weights}) respectively.
Then 

\begin{equation}
\left\{ \begin{array}{lc}
\partial_{t}\widetilde{x}_{N}(t,s)=\displaystyle\int_{I^{d}}\widetilde{m}_{N}(t,s_{\ast})\mathbf{a}(\widetilde{x}_{N}(t,s_{\ast})-\widetilde{x}_{N}(t,s))ds_{\ast},\\[3mm]
\partial_{t}\widetilde{m}_{N}(t,s)=N^{d}\displaystyle\int_{\mathcal{Q}_{\left\lfloor s_{1}N\right\rfloor ,\ldots,\left\lfloor s_{d}N\right\rfloor }}\Psi(s_{\ast},\widetilde{x}_{N}(t,\cdot),\widetilde{m}_{N}(t,\cdot))ds_{\ast}. 
\end{array}\right.\label{eq:labled equation governing xN,mN}
\end{equation}
\end{prop}
\begin{proof}
We start with the equation for $\widetilde{x}_{N}(t,s)$. Fix $s\in\mathcal{Q}_{\mathbf{i}_0}$.
Then 
\[
\partial_{t}\widetilde{x}_{N}(t,s)=\dot{x}_{\sigma(\mathbf{i}_0)}(t)=\frac{1}{N^{d}}\stackrel[j=1]{N^d}{\sum}m_{j}(t)\mathbf{a}(x_{j}(t)-x_{\sigma(\mathbf{i}_0)}(t))=\frac{1}{N^{d}}\stackrel[j=1]{N^d}{\sum}m_{j}(t)\mathbf{a}(x_{j}(t)-\widetilde{x}_{N}(t,s)).
\]
On the other hand 
\begin{align*}
\int_{I^{d}}\widetilde{m}_{N}(t,s_{\ast})\mathbf{a}(\widetilde{x}_{N}(t,s_{\ast})-\widetilde{x}_{N}(t,s))ds_{\ast}&=\int_{I^{d}}\underset{\mathbf{i}\in\mathcal{\mathcal{M}}}{\sum}m_{\sigma\left(\mathbf{i}\right)}\mathbf{1}_{\mathcal{Q}_{\mathbf{i}}}(s_{\ast})\mathbf{a}\left(\underset{\mathbf{k}\in\mathcal{\mathcal{M}}}{\sum}x_{\sigma\left(\mathbf{k}\right)}(t)\mathbf{1}_{\mathcal{Q}_{\mathbf{k}}}(s_{\ast})-\widetilde{x}_{N}(t,s)\right)ds_{\ast}\\
&=\frac{1}{N^{d}}\underset{\mathbf{i}\in\mathcal{\mathcal{M}}}{\sum}m_{\sigma\left(\mathbf{i}\right)}\mathbf{a}\left(x_{\sigma\left(\mathbf{i}\right)}(t)-\widetilde{x}_{N}(t,s)\right)\\&=\frac{1}{N^{d}}\stackrel[j=1]{N^d}{\sum}m_{j}(t)\mathbf{a}\left(x_{j}(t)-\widetilde{x}_{N}(t,s)\right).
\end{align*}
The equation for $\widetilde{m}_{N}$ is due to the following identities
\[
\partial_{t}\widetilde{m}_{N}(t,s)=\underset{\mathbf{i}\in\mathcal{\mathcal{M}}}{\sum}\dot{m}_{\sigma\left(\mathbf{i}\right)}\mathbf{1}_{\mathcal{Q}_{\mathbf{i}}}(s)=N^{d}\int_{\mathcal{Q}_{\left\lfloor s_{1}N\right\rfloor ,\ldots,\left\lfloor s_{d}N\right\rfloor }}\Psi(s_{\ast},\widetilde{x}_{N}(t,\cdot),\widetilde{m}_{N}(t,\cdot))ds_{\ast}.
\]
\end{proof}
\subsection{Well posedness for $d>1$}
The point which requires most care for the proof of well posedness is point 2. in Lemma \ref{Estimate for psi}. Let us first state the assumptions we impose on the initial data and the other functions involved. 
\begin{itemize}
 \item[\textbf{A1}]  $d>1, \mathbf{a}(0)=0$  and $\mathbf{a}\in\mathrm{Lip}(\mathbb{R}^{d})$
with $L\coloneqq\mathrm{Lip}(\mathbf{a})$. 

\item[\textbf{A2}] i. $m^{0}\in L^{\infty}(I^d)$, $\int_{I^d} m_{0}(s)ds=1$ and $\frac{1}{M}\leq m^{0}\leq M$ for some $M>1$.  

ii. $\left|x^{0}\right|\leq X$ for some $X>0$ and $x^{0}:I^d\rightarrow\mathbb{R}^{d}$ is bi-Lipschitz, i.e. there is some $L_0>0$ such that for all $s_1,s_2\in I^{d}$
\[
\frac{1}{L_0}\left|s_{1}-s_{2}\right|\leq\left|x^{0}(s_{1})-x^{0}(s_{2})\right|\leq L_0\left|s_{1}-s_{2}\right|.
\]

 \item[\textbf{A3}]   i. $\mathbf{s}:\mathbb{R}^d\rightarrow \mathbb{R}$ is a measurable odd function ($\mathbf{s}(-x)=-\mathbf{s}(x)$).
 
ii. There is some locally $L^1$ function    $\mathbf{S}\geq0$ such that 
 \[ \left|\mathbf{s}(x_{1})-\mathbf{s}(x_{2})\right|\leq { \mathbf{S}(x_2)}\,
 \left|x_{1}-x_{2}\right|,\ x_{1},x_{2}\in\mathbb{R}^{d}.\]
\end{itemize}
\begin{rem}
The assumption \textbf{A2} that $x^0$ is bi-Lipschitz is strictly stronger than the assumption that it is 1-1 when $d=1$.     
\end{rem}
\begin{rem}
If $d>1$, then $\mathbf{s}(x)\coloneqq\left\{ \begin{array}{c}
\frac{x}{\left\Vert x\right\Vert },\ x\neq0\\
0,\ x=0.
\end{array}\right.$ is a particular example of hypothesis \textbf{A3 }as can be seen
from through the following elementary inequalities
\begin{align*}
\left|\frac{x_{1}}{\left|x_{1}\right|}-\frac{x_{2}}{\left|x_{2}\right|}\right|=\left|\frac{x_{1}\left|x_{2}\right|-\left|x_{1}\right|x_{2}}{\left|x_{1}\right|\left|x_{2}\right|}\right|=&\left|\frac{x_{1}\left(\left|x_{2}\right|-\left|x_{1}\right|\right)+\left|x_{1}\right|\left(x_{1}-x_{2}\right)}{\left|x_{1}\right|\left|x_{2}\right|}\right|\\
\leq&\frac{\left|\left|x_{1}\right|-\left|x_{2}\right|\right|}{\left|x_{2}\right|}+\frac{\left|x_{1}-x_{2}\right|}{\left|x_{2}\right|}\leq\frac{2\left|x_{1}-x_{2}\right|}{\left|x_{2}\right|}.
\end{align*}
Furthermore, it is clear that the condition i. in \textbf{A3} is more general than the assumption $ \mathbf{s}\in\mathrm{Lip}(\mathbb{R}^{d})$.
\end{rem}
\begin{lem}
\label{Estimate for psi d>1} Let hypotheses  \textbf{A1}-\textbf{A3} hold. Suppose that\\ 
$\bullet$ There is some $\overline{X}$ such that 
$\sup_{[0,T]\times I}\left|x_i\right|\leq\overline{X},$\ $i=1,2.$\\
$\bullet$ The function $x_{2}$ is bi-Lipschitz in the labeling variable, i.e. there is some $C>1$ such that for 
all $(t,s,s_{\ast})\in[0,T]\times I^d\times I^d$ it holds that 
\[\frac{1}{C}\left|s-s_{\ast}\right|\leq\left|x_{2}(t,s)-x_{2}(t,s_{\ast})\right|\leq C\left|s-s_{\ast}\right|,\ i=1,2.\]
Then 
\begin{align*}
&\int_{I^d}\left|\Psi(s,x_{1}(t,\cdot),m(t,\cdot))-\Psi(s,x_{2}(t,\cdot),m(t,\cdot))\right|ds\leq \mathcal{C}\,\underset{I^d}{\sup}\left|x_{1}(t,\cdot)-x_{2}(t,\cdot)\right|,    
\end{align*}
with {$\mathcal{C}=L\,\overline{M}(6+8\,\overline{X}\,C)$.} 
\end{lem}
\textit{Proof}. 
Setting $\mathbf{a}_i(s,s_\ast,s_{\ast\ast}):=\mathbf{a}(x_i(s_{\ast\ast})-x_i(s))+\mathbf{a}(x_i(s_{\ast\ast})-x_i(s_{\ast}))$, $i=1,2$, we can  estimate A2
\begin{align*}
\left|\iint_{I^{2d}}m(s_{\ast})m(s_{\ast\ast})(\mathbf{a}_1(s,s_\ast,s_{\ast\ast})\mathbf{s}(x_{1}(s)-x_{1}(s_{\ast}))ds_{\ast}ds_{\ast\ast}-\mathbf{a}_2(s,s_\ast,s_{\ast\ast})\mathbf{s}(x_{2}(s)-x_{2}(s_{\ast})))ds_{\ast}ds_{\ast\ast}\right|\hspace{3.0 cm}\\
\leq J_{1}(t,s)+J_{2}(t,s),\hspace{14.3 cm}
\end{align*}
where we define
\begin{align*}
J_1(t,s)=L\,\iint_{I^{2d}} &m(s_{\ast})m(s_{\ast\ast})\left(2\left|x_{1}(s_{\ast\ast})-x_{2}(s_{\ast\ast})\right|+\left|x_{1}(s_{\ast})-x_{2}(s_{\ast})\right|+\left|x_{1}(s)-x_{2}(s)\right|\right)\\
&(|\mathbf{s}(x_1(s)-x_1(s_{\ast}))|+|\mathbf{s}(x_2(s)-x_2(s_{\ast}))|)\,ds_{\ast}ds_{\ast\ast},
\end{align*}
and
\[
J_2(t,s)=\iint_{I^{2d}}m(s_{\ast})m(s_{\ast\ast})\left|\mathbf{a}_2(s,s_\ast,s_{\ast\ast})\right| \left|\mathbf{s}(x_{1}(s)-x_{1}(s_{\ast}))-\mathbf{s}(x_{2}(s)-x_{2}(s_{\ast}))\right|ds_{\ast}ds_{\ast\ast}.
\]
We start with the estimate on $J_{2}(t,s)$. Using assumption $\mathbf{A3}$, we have that
\begin{align*}
&\left|\iint_{I^{2d}}m(s_{\ast})m(s_{\ast\ast})\mathbf{a}_{2}(s,s_{\ast},s_{\ast\ast})\mathbf{s}(x_{1}(s)-x_{1}(s_{\ast}))ds_{\ast}ds_{\ast\ast}\right.\\
&\left.-\iint_{I^{2d}}m(s_{\ast})m(s_{\ast\ast})\mathbf{a}_{2}(s,s_{\ast},s_{\ast\ast})\mathbf{s}(x_{2}(s)-x_{2}(s_{\ast}))ds_{\ast}ds_{\ast\ast}\right|\\
&\leq L\sup_{[0,T]\times I^d}\iint_{I^{2d}}(2|x_2(s_{\ast\ast})|+|x_2(s)|+|x_2(s_\ast)|)\,m(s_{\ast})m(s_{\ast\ast})\left|\mathbf{s}(x_{1}(s)-x_{1}(s_{\ast}))-\mathbf{s}(x_{2}(s)-x_{2}(s_{\ast}))\right|ds_{\ast}ds_{\ast\ast}\\
&\leq 4\,L\,\overline{X}\,\iint_{I^{2d}} m(s_{\ast})m(s_{\ast\ast})\, \mathbf{S}(x_2(s)-x_2(s_\ast))\,(|x_1(s)-x_2(s)|+|x_1(s_\ast)-x_2(s_\ast)|)\,ds_\ast\,ds_{\ast\ast}\\
&\leq8\,L\,\overline{X}\,\overline{M}\,\underset{I^d}{\sup}\left|x_{1}(t,\cdot)-x_{2}(t,\cdot)\right|\;
\sup_{s\in I^d} \int_{I^{d}}\mathbf{S}(x_2(s)-x_2(s_\ast))\,ds_{\ast}.
\end{align*}
 From the bi-Lipschitz assumption on $x_2$ we have $|(\nabla_s x_{2})^{-1}|\leq C$ so that 
\begin{align*}
\int_{I^{d}}\mathbf{S}(x_2(s)-x_2(s_\ast))\,ds_{\ast}
\leq   C\,\int_{I^{d}}\mathbf{S}(x_2(s)-y)\,dy\leq C\,\|S\|_{L^1(K)},
\end{align*}
for some compact set $K\subset \mathbb{R}^d$.  Therefore 
\begin{align*}
\int_{I^d}\ m(t,s)\left |J_2{(t,s)}\right|ds\leq 
8\,L\,\overline{X}\,\overline{M}\,C\,\|\mathbf{S}\|_{L^1(K)}\,\underset{I^d}{\sup}\left|x_{1}(t,\cdot)-x_{2}(t,\cdot)\right|.
\end{align*}
 The estimate for $J_1{(t,s)}$ follows in a similar way,
\begin{align*}
&\int_{I^d} m(t,s)\left|J_{1}(t,s)\right|ds\\
&\qquad\leq 3L\,\underset{I^d}{\sup}\left|x_{1}(t,\cdot)-x_{2}(t,\cdot)\right|\,\int_{I^{2d}}m(s)\,m(s_\ast)\, (|\mathbf{s}(x_1(s)-x_1(s_\ast))|+|\mathbf{s}(x_2(s)-x_2(s_\ast))|)\,ds\,ds_\ast\\
&\qquad \leq 6\,L\,\overline{M}\|\mathbf{S}\|_{L^1(K)}\,\underset{I^d}{\sup}\left|x_{1}(t,\cdot)-x_{2}(t,\cdot)\right|.
\end{align*}
\begin{flushright}
$\square$
\par\end{flushright}
\subsection{Graph limit for $d>1$}
Consider 
\begin{align*}
\xi_{N}(t)\coloneqq\left\Vert \widetilde{x}_{N}(t,\cdot)-x(t,\cdot)\right\Vert _{L^{1}(I^{d})},\ \zeta_{N}(t)\coloneqq\left\Vert \widetilde{m}_{N}(t,\cdot)-m(t,\cdot)\right\Vert _{L^{1}(I^{d})}.    
\end{align*}
The symbol $\lesssim$
stands for inequality up to a constant which may depend only on $L,L_0,M,X,T,\mathbf{S},A2\mathbf{S}_{\infty}$.
\begin{thm}
\label{the graph limit: main thm high dim} Let the hypotheses \textbf{A1}-\textbf{A3} hold. 
Let $(x,m)\in C^{1}\left([0,T];C(I^{d})\right)\oplus C^{1}\left([0,T];L^{\infty}(I^{d})\right)$
be the solution to Equation \eqref{eq:COUPLED EQUATION}. Assume that $\mathbf{S}$ is chosen so that the system \eqref{opinion dynamics sec2} has a well defined flow. Let $(\mathbf{x}_{N},\mathbf{m}_{N})\in C^{1}\left(\left[0,T\right];\mathbb{R}^{dN^d}\times{\mathbb{R}}^{N^d}\right)$
be the solution to the system \eqref{opinion dynamics sec2}. Then
\[
\left\Vert \widetilde{x}_{N}-x\right\Vert _{C([0,T];L^{1}(I^{d}))}+\left\Vert \widetilde{m}_{N}-m\right\Vert _{C([0,T];L^{1}(I^{d}))}\underset{N\rightarrow\infty}{\rightarrow}0.
\]
\end{thm}
\textit{Proof}. \textbf{Step 1}. \textbf{Time derivative of} $\zeta_{N}(t)$.
The computation of the time derivative of $\zeta_{N}$ is essentially identical to the case $d=1$, but we include it for clarity. For readability
we set $\mathbf{i}_{N}(s)\coloneqq\left(\left\lfloor s_{1}N\right\rfloor ,\ldots,\left\lfloor s_{d}N\right\rfloor \right)$. 

\begin{align*}
\dot{\zeta_{N}}(t)\leq\int_{I^{d}}\left|N^{d}\int_{\mathcal{Q}_{\mathbf{i}_{N}(s)}}\Psi(s_{\ast},\widetilde{x}_{N}(t,\cdot),\widetilde{m}_{N}(t,\cdot))ds_{\ast}-\Psi(s,x(t,\cdot),m(t,\cdot))\right|ds\\
\leq\int_{I^{d}}\underset{\coloneqq h_{N}(t,s)}{\underbrace{\left|N^{d}\int_{\mathcal{Q}_{\mathbf{i}_{N}(s)}}\Psi(s_{\ast},\widetilde{x}_{N}(t,\cdot),\widetilde{m}_{N}(t,\cdot))-\Psi(s_{\ast},x(t,\cdot),m(t,\cdot))ds_{\ast}\right|}}ds\\
+\int_{I^{d}}\underset{\coloneqq g_{N}(t,s)}{\underbrace{\left|N^{d}\int_{\mathcal{Q}_{\mathbf{i}_{N}(s)}}\Psi(s_{\ast},x(t,\cdot),m(t,\cdot))ds_{\ast}-\Psi(s,x(t,\cdot),m(t,\cdot))\right|}}ds.
\end{align*}
By Lebesgue differentiation theorem, for each $t\in[0,T]$ it holds
that $g_{N}(t,s){\rightarrow}0$ as $N\rightarrow\infty$ pointwise
a.e. $s\in I^{d}$. In addition, $\left\Vert x\right\Vert _{C([0,T]\times I^{d})},\left\Vert m\right\Vert _{C([0,T];L^{\infty}(I^{d}))}$
are bounded which implies that $g_{N}(t,s)$ is uniformly bounded
(with respect to $N$) so that by the dominated convergence theorem
we find that for each $t\in[0,T]$ it holds that
$\left\Vert g_{N}(t,\cdot)\right\Vert _{1}\underset{N\rightarrow\infty}{\rightarrow}0.$
We now estimate the first integral as
\begin{align*}
\left\Vert h_{N}(t,\cdot)\right\Vert _{1}=&\,N^{d}\int_{I^{d}}\left|\int_{\mathcal{Q}_{\mathbf{i}_{N}(s)}}\Psi(s_{\ast},\widetilde{x}_{N}(t,\cdot),\widetilde{m}_{N}(t,\cdot))-\Psi(s_{\ast},x(t,\cdot),m(t,\cdot))ds_{\ast}\right|ds\\=&\,N^{d}\underset{\mathbf{i}\in\mathcal{M}}{\sum}\int_{\mathcal{Q}_{\mathbf{i}}}\left|\int_{\mathcal{Q}_{\mathbf{i}_{N}(s)}}\Psi(s_{\ast},\widetilde{x}_{N}(t,\cdot),\widetilde{m}_{N}(t,\cdot))-\Psi(s_{\ast},x(t,\cdot),m(t,\cdot))ds_{\ast}\right|ds\\\leq& \,N^{d}\underset{\mathbf{i}\in\mathcal{M}}{\sum}\int_{\mathcal{Q}_{\mathbf{i}}}\int_{\mathcal{Q}_{\mathbf{i}_{N}(s)}}\left|\Psi(s_{\ast},\widetilde{x}_{N}(t,\cdot),\widetilde{m}_{N}(t,\cdot))-\Psi(s_{\ast},x(t,\cdot),m(t,\cdot))\right|ds_{\ast}ds\\=&\,N^{d}\underset{\mathbf{i}\in\mathcal{M}}{\sum}\int_{\mathcal{Q}_{\mathbf{i}}}\int_{\mathcal{Q}_{\mathbf{i}}}\left|\Psi(s_{\ast},\widetilde{x}_{N}(t,\cdot),\widetilde{m}_{N}(t,\cdot))-\Psi(s_{\ast},x(t,\cdot),m(t,\cdot))\right|ds_{\ast}ds\\=&\,\underset{\mathbf{i}\in\mathcal{M}}{\sum}\int_{\mathcal{Q}_{\mathbf{i}}}\left|\Psi(s_{\ast},\widetilde{x}_{N}(t,\cdot),\widetilde{m}_{N}(t,\cdot))-\Psi(s_{\ast},x(t,\cdot),m(t,\cdot))\right|ds_{\ast}\\=&\int_{I^{d}}\left|\Psi(s_{\ast},\widetilde{x}_{N}(t,\cdot),\widetilde{m}_{N}(t,\cdot))-\Psi(s_{\ast},x(t,\cdot),m(t,\cdot))\right|ds_{\ast}.\\
\end{align*}
This in turn leads to the following bound
\begin{align*}
\left\Vert h_{N}(t,\cdot)\right\Vert _{1}
\leq&\int_{I^{d}}\left|\Psi(s_{\ast},\widetilde{x}_{N}(t,\cdot),\widetilde{m}_{N}(t,\cdot))-\Psi(s_{\ast},\widetilde{x}_{N}(t,\cdot),m(t,\cdot))\right|ds_{\ast}\\
&+\int_{I^{d}}\left|\Psi(s_{\ast},\widetilde{x}_{N}(t,\cdot),m(t,\cdot))-\Psi(s_{\ast},x(t,\cdot),m(t,\cdot))\right|ds_{\ast}\\=&\int_{I^{d}}\left|\Psi(s,\widetilde{x}_{N}(t,\cdot),\widetilde{m}_{N}(t,\cdot))-\Psi(s,\widetilde{x}_{N}(t,\cdot),m(t,\cdot))\right|ds\\&+\int_{I^{d}}\left|\Psi(s,\widetilde{x}_{N}(t,\cdot),m(t,\cdot))-\Psi(s,x(t,\cdot),m(t,\cdot))\right|ds.
\end{align*}
At this stage we state and prove the following simple adaption of Lemma \ref{Estimate for psi d>1}. We precisely need the following result.
\begin{lem}
It holds that 
\[
i.\int_{I^{d}}\left|\Psi(s,\widetilde{x}_{N}(t,\cdot),\widetilde{m}_{N}(t,\cdot))-\Psi(s,\widetilde{x}_{N}(t,\cdot),m(t,\cdot))\right|ds\lesssim\left\Vert \widetilde{m}_{N}(t,\cdot)-m(t,\cdot)\right\Vert _{1}.
\]
and 
\[
ii.\int_{I^{d}}\left|\Psi(s,\widetilde{x}_{N}(t,\cdot),m(t,\cdot))-\Psi(s,x(t,\cdot),m(t,\cdot))\right|ds\lesssim\left\Vert \widetilde{x}_{N}(t,\cdot)-x(t,\cdot)\right\Vert _{1}.
\] 
\label{L1 contraction estimate}
\end{lem}
\textit{Proof}. Point i. follows from the same argument of 1. in Lemma \ref{Estimate for psi}, so let us concentrate on point ii. We estimate 
\begin{align*}
&\left|\iint_{I^{2d}}m(s_{\ast})m(s_{\ast\ast})\mathbf{a}_{N}(s,s_{\ast},s_{\ast\ast})\mathbf{s}(\widetilde{x}_{N}(s)-\widetilde{x}_{N}(s_{\ast}))ds_{\ast}ds_{\ast\ast}\right.\\    
&\left.-\iint_{I^{2d}}m(s_{\ast})m(s_{\ast\ast})\mathbf{a}(s,s_{\ast},s_{\ast\ast})\mathbf{s}(x(s)-x(s_{\ast}))ds_{\ast}ds_{\ast\ast}\right|\\
\lesssim& \iint_{I^{2d}}m(s_{\ast})m(s_{\ast\ast})\left(2\left|\widetilde{x}_{N}(s_{\ast\ast})-x(s_{\ast\ast})\right|+\left|\widetilde{x}_{N}(s_{\ast})-x(s_{\ast})\right|+\left|\widetilde{x}_{N}(s)-x(s)\right|\right)ds_{\ast}ds_{\ast\ast}\\
&+\iint_{I^{2d}}m(s_{\ast})m(s_{\ast\ast})\left|\mathbf{a}(x(s_{\ast\ast})-x(s))+\mathbf{a}(x(s_{\ast\ast})-x(s_{\ast}))\right|\\
&\times\left|\mathbf{s}(\widetilde{x}_{N}(s)-\widetilde{x}_{N}(s_{\ast}))-\mathbf{s}(x(s)-x(s_{\ast}))\right|ds_{\ast}ds_{\ast\ast}\\&\coloneqq J_{1}(t,s)+J_{2}(t,s).
\end{align*}
We estimate seperately $\int_{I^d}m(t,s)\left|J_{1}(t,s)\right|ds$
and $\int_{I^d}m(t,s)\left|J_{2}(t,s)\right|ds$. It is straightforward
to check
\begin{equation}
\int_{I^{d}}m(t,s)\left|J_{1}(t,s)\right|ds\lesssim\left\Vert \widetilde{x}_{N}(t,\cdot)-x(t,\cdot)\right\Vert _{1}.\label{eq:J1 es-1}
\end{equation}
We further estimate 
\begin{align*}
\int_{I^{d}}m(t,s)\left|J_{2}(t,s)\right|ds&\lesssim\iint_{I^{2d}}\mathbf{S}(x(t,s)-x(t,s_{\ast}))\left(\left|\widetilde{x}_{N}(t,s)-x(t,s)\right|+\left|\widetilde{x}_{N}(t,s_{\ast})-x(t,s_{\ast})\right|\right)ds_{\ast}ds\\ &\leq2\left\Vert \mathbf{S}\right\Vert _{L^{1}(K)}\left\Vert \widetilde{x}_{N}(t,\cdot)-x(t,\cdot)\right\Vert _{1},
\end{align*}
for some compact set $K\subset \mathbb{R}^d$.
\begin{flushright}
    $\square$
\end{flushright}
\textbf{Step 2. The time derivative of $\xi_{N}(t)$.} The time derivative
of $\dot{\xi_{N}}(t)$ is mastered exactly as in {[}1{]}. Following
the argument in {[}1{]} one finds that 
\begin{equation}
\dot{\xi}_{N}(t)\lesssim\xi_{N}(t)+\zeta_{N}(t).\label{eq:-3-1}
\end{equation}
\textbf{Step 3. Conclusion. } The combination of Inequality (\ref{eq:-3-1}) and Lemma \ref{L1 contraction estimate} yields 
\begin{align*}
\dot{\xi}_{N}(t)+\dot{\zeta}_{N}(t)\lesssim\xi_{N}(t)+\zeta_{N}(t)+\frac{1}{2}\left\Vert g_{N}(t,\cdot)\right\Vert _{1},   \end{align*}
i.e.
\begin{align*}
\xi_{N}(t)+\zeta_{N}(t)\lesssim\xi_{N}(0)+\zeta_{N}(0)+\int_{0}^{t}(\xi_{N}(\tau)+\zeta_{N}(\tau))d\tau+\int_{0}^{T}\left\Vert g_{N}(\tau,\cdot)\right\Vert _{1}d\tau.    
\end{align*}
Applying Gronwall's lemma entails 

\[
\xi_{N}(t)+\zeta_{N}(t)\leq C\left(\xi_{N}(0)+\zeta_{N}(0)+\int_{0}^{T}\left\Vert g_{N}(\tau,\cdot)\right\Vert _{1}d\tau+o_{N}(1)T\right)\exp\left(Ct\right).
\]
Since $\left\Vert g_{N}(\tau,\cdot)\right\Vert _{2}^{2}$ is uniformly
bounded, by dominated convergence
$\int_{0}^{T}\left\Vert g_{N}(\tau,\cdot)\right\Vert _{1}d\tau\underset{N\rightarrow\infty}{\rightarrow}0$,
which concludes the proof. 
\begin{flushright}
$\square$
\par\end{flushright}
\begin{rem}
Note that Theorem \ref{the graph limit: main thm high dim} proves convergence  with respect to the $L^1$ norm, whereas Theorem \ref{the graph limit: main thm} proves convergence with respect to the $L^2$ norm. This minor difference is because when $d=2$ the $L^2$ norm of $\frac{1}{\left|s\right|}$ blows up, which prevents getting the inequality ii. in Lemma \ref{L1 contraction estimate}. Notice that for any $d\geq 3$, the $L^2$ approach is perfectly valid. 
\end{rem}
\begin{rem}
Essentially the same argument of Theorem \ref{weak mean field} allows one to conclude the weak mean field limit from the graph limit in higher dimension.  
\end{rem}
\subsubsection*{Acknowledgments}
IBP and JAC were supported by the EPSRC grant number EP/V051121/1.
This work was also supported by the Advanced Grant Nonlocal-CPD (Nonlocal PDEs for Complex Particle Dynamics: Phase Transitions, Patterns and Synchronization) of the European Research Council Executive Agency (ERC) under the European Union's Horizon 2020 research and innovation programme (grant agreement No. 883363). We are grateful for the comments of the anonymous referee.

\end{document}